\documentclass[journal]{IEEEtran}
\usepackage[latin9]{inputenc}
\usepackage{float}
\usepackage{amsmath}
\usepackage{amssymb}
\usepackage{graphicx}
\usepackage{color}
\usepackage{url}
\usepackage[nospace,compress]{cite}

\makeatletter
\floatstyle{ruled}
\newfloat{algorithm}{tbp}{loa}
\providecommand{\algorithmname}{Algorithm}
\floatname{algorithm}{\protect\algorithmname}
\usepackage{amsbsy}\usepackage{epsfig}
\usepackage{subcaption}

\newcommand{\beq}{\begin{equation}}
\newcommand{\eeq}{\end{equation}}
\def\bv{\mbox{\boldmath $v$}}

\def\bu{\mbox{\boldmath $u$}}

\def\bh{\mbox{\boldmath $h$}}

\def\bv{\mbox{\boldmath $v$}}
\def\bw{\mbox{\boldmath $w$}}

\def\bx{\mbox{\boldmath $x$}}

\def\bs{\mbox{\boldmath $s$}}

\def\bL{\mbox{\boldmath $L$}}

\def\mLambda{\mbox{$\mathbf{\Lambda}$}}
\def\mSigma{\mbox{$\mathbf{\Sigma}$}}
\def\mTheta{\mbox{$\mathbf{\Theta}$}}
\def\ma{\mbox{$\mathbf{a}$}}
\def\mA{\mbox{$\mathbf{A}$}}
\def\mB{\mbox{$\mathbf{B}$}}

\def\mD{\mbox{$\mathbf{D}$}}

\def\mg{\mbox{$\mathbf{g}$}}
\def\mLambda{\mbox{$\mathbf{\Lambda}$}}
\def\mSigma{\mbox{$\mathbf{\Sigma}$}}
\def\mF{\mbox{$\mathbf{F}$}}

\def\mJ{\mbox{$\mathbf{J}$}}
\def\mI{\mbox{$\mathbf{I}$}}
\def\mL{\mbox{$\mathbf{L}$}}
\def\mY{\mbox{$\mathbf{Y}$}}
\def\mQ{\mbox{$\mathbf{Q}$}}
\def\mZ{\mbox{$\mathbf{Z}$}}
\def\mR{\mbox{$\mathbf{R}$}}
\def\mS{\mbox{$\mathbf{S}$}}
\def\mT{\mbox{$\mathbf{T}$}}
\def\mP{\mbox{$\mathbf{P}$}}
\def\mU{\mbox{$\mathbf{U}$}}
\def\mV{\mbox{$\mathbf{V}$}}

\def\my{\mbox{$\mathbf{y}$}}
\def\mX{\mbox{$\mathbf{X}$}}
\def\mx{\mbox{$\mathbf{x}$}}
\def\mW{\mbox{$\mathbf{W}$}}

\newcommand{\ds}{\displaystyle}

\def\tr{\, \mbox{trace} }

\newtheorem{theorem}{Theorem}
\newtheorem{proposition}{Proposition}
\newtheorem{definition}{Definition}

\newenvironment{proof}[1][Proof]{\noindent \textbf{#1.} }{\qedsymbol}
\newcommand{\qedsymbol}{\hspace{\fill}\rule{1.5ex}{1.5ex}}

\makeatother

\begin{document}

\title{On the Graph Fourier Transform\\ for Directed Graphs\vspace{.2cm}}

\author{Stefania Sardellitti,~\IEEEmembership{Member,~IEEE}, Sergio Barbarossa,~\IEEEmembership{Fellow,~IEEE}, and Paolo Di Lorenzo,~\IEEEmembership{Member,~IEEE} \\
\thanks{S. Sardellitti and S. Barbarossa are with Sapienza University of Rome, DIET Dept., Via Eudossiana 18, 00184 Rome, Italy (e-mail: stefania.sardellitti@uniroma1.it, sergio.barbarossa@uniroma1.it). P. Di Lorenzo is with the Dept. of Engineering, University of Perugia, Via G. Duranti 93, 06125 Perugia, Italy (e-mail: paolo.dilorenzo@unipg.it).
This work has been
supported by TROPIC Project, Nr. ICT-318784.
The work of P. Di Lorenzo was funded by the ``Fondazione Cassa di Risparmio di Perugia''.  Matlab code to implement the algorithms proposed in this paper is available at https://sites.google.com/site/stefaniasardellitti/code-supplement} \vspace{-.63cm}}

\maketitle

\begin{abstract}
The analysis of signals defined over a graph is relevant in many applications, such as social and economic networks, big data or biological networks, and so on. A key tool for analyzing these signals is the so called Graph Fourier Transform (GFT). Alternative definitions of GFT have been suggested in the literature, based on the eigen-decomposition of either the graph Laplacian or adjacency matrix. In this paper, we address the general case of directed graphs and we propose an alternative approach that builds the graph Fourier basis as the set of orthonormal vectors that minimize a continuous extension of the graph cut size, known as the Lov\'{a}sz extension. To cope with the non-convexity of the problem, we propose two alternative iterative optimization methods, properly devised for handling orthogonality constraints. Finally,  we extend the method to minimize  a continuous relaxation of the   balanced cut size. The formulated problem is again non-convex and we propose an efficient solution method based on an explicit-implicit gradient algorithm.
\end{abstract}


\IEEEpeerreviewmaketitle

\section{Introduction}

Graph signal processing (GSP) has attracted a lot of interest in the last years because of its many potential
applications, from  social and economic networks to smart grids, gene regulatory  networks,  and so on.
GSP represents a promising tool for the representation,   processing and analysis of  complex networks, where discrete signals are defined on the vertices of a (possibly weighted) graph.
Many works in the recent literature attempt to extend the classical discrete signal processing (DSP) theory from
time signals or images to signals defined over the vertices of a graph by introducing  the basic concepts of graph-based filtering \cite{Moura2013, Narang2012, Narang2013}, graph-based transforms  \cite{Moura2014, Shuman2013, Hammond, NarangOrtega}, sampling and uncertainty principle \cite{Pesenson2008, Agaskar, Tsit_Barb_PDL, Tsit_Barb, Chen_Varma}.
A central role in GSP is played by the spectral analysis of graph signals, which is based on the introduction of the so
called  Graph Fourier Transform (GFT).
Alternative definitions of GFT have been introduced
see, e.g., \cite{Shuman2013},\cite{Moura2014},\cite{Pesenson2008},\cite{Zhu2012},\cite{Singh2016}, each of them coming from different motivations, like building a basis with minimal variation, filtering signals defined over graphs, etc.
Two basic approaches have been suggested.
The first one is rooted on spectral graph theory and it uses the graph-Laplacian as the central
unit, see e.g. \cite{Shuman2013} and the references therein. This approach applies to undirected
graphs and the Fourier basis is constituted by the eigenvectors of the graph Laplacian, which represent the
basis that minimizes the $l_2$-norm graph total variation.
This approach is well motivated on undirected graphs where the minimization of the $\ell_2$-norm total variation is equivalent to minimizing the quadratic form built on the Laplacian matrix. Hence an orthonormal basis minimizing the $\ell_2$-norm total variation leads to the eigenvectors of the Laplacian matrix. However, these properties do not hold anymore in the directed graph case.
An alternative approach, valid for the more general and challenging case of directed graphs, was proposed in \cite{Moura2013},\cite{Moura2014}. That method builds on the Jordan decomposition of the adjacency matrix, and defines the associated generalized eigenvectors as the GFT basis. This second method is rooted on the association of the graph adjacency matrix with the signal shift operator, which is at the basis of all shift-invariant linear filtering methods for graph signals \cite{Puschel_1},\cite{Puschel_2}. This approach  paved the way to the algebraic signal processing framework. However, the GFT definition proposed in \cite{Moura2014} raises some  important issues requiring  further investigation. First, the basis vectors are linearly  independent, but in general they are not orthogonal, so that the resulting transform is not unitary and then it does not preserve scalar products. Second, the total variation
introduced in \cite{Moura2014},  does not respect some desirable properties, for example, it does not guarantee that a constant graph signal has zero total variation \cite{Mallat}, \cite{Lozes}. Finally, the numerical computation of the Jordan decomposition often incurs  into well-known numerical instabilities, even for moderate size matrices  \cite{Golub76}, although  alternative decomposition methods have been recently suggested to tackle these instability issues \cite{girault}.
In some applications, one of the major motivations for using the GFT is the analysis of graph signals that exhibit clustering properties, i.e. signals that are smooth within subsets of highly interconnected nodes (clusters), while they can vary arbitrarily across different clusters.  In such cases, the GFT of these signals is typically sparse and its sparsity carries relevant information on the data under analysis. These signals are said to be band-limited, in analogy with what happens to smooth time signals. Within the machine learning context, GSP can play a key role in unsupervised and semi-supervised learning, as suggested in \cite{GaddeAnisOrtega}, \cite{AnisGAO15}. In these applications, the input is a point cloud and the goal is to detect clusters, either without or with limited supervision. Graph-based methods tackle these problems by associating a graph to the point cloud, where the vertices are the points themselves, whereas edges between pairs of points are established if two points are sufficiently close. The goal of clustering/classification is to associate a different label to each cluster. If we look at these labels as a signal defined over the points (vertices), this signal is band-limited by construction \cite{GaddeAnisOrtega}, \cite{AnisGAO15}.

In this paper, we propose a novel alternative approach to build the GFT basis for the general case of directed graphs. Rather than starting from the decomposition of one of the graph matrix descriptors, either adjacency or Laplacian, we start identifying an objective function to be minimized and then we build an orthogonal matrix that minimizes that objective function. More specifically, we choose as objective function the graph cut size, as its minimization leads to identifying clusters. We consider the general case of directed graphs, which subsumes the undirected graphs as a particular case. The cut function is a set function and its minimization is NP-hard, however exploiting the sub-modularity property of the cut size, it has been shown that there exists a lossless convex relaxation of the cut size, named its Lov\'{a}sz extension \cite{Lovasz1983}, \cite{Bach2013}, whose minimization preserves the optimality of the solution of the original non-convex problem. Interestingly,
the Lov\'{a}sz extension of the cut size gives rise to an alternative definition of total variation of a graph signal that captures the edges' directivity.  Furthermore, in the case of undirected graphs, the Lov\'{a}sz extension reduces to the  $l_1$ norm total variation of a graph signal, which  represents the discrete counterpart of the total variation of continuous-time signals, which plays  a fundamental role in the continuous time Fourier Transform, see, e.g., \cite{Mallat},\cite{Zhu2012}.
We define the GFT basis as the set of orthonormal vectors that minimize the Lov\'{a}sz extension of the cut size. Unfortunately, even though the objective function is convex, the resulting problem is non-convex, because of the orthogonality constraint imposed on the basis vectors. Thus, to find a (possibly local) solution of the problem in an efficient manner, we exploit two recently developed methods that are specifically tailored to handle non-convex orthogonality constraints, namely, the splitting orthogonality constraints (SOC) method \cite{Lai2014}, and the proximal alternating minimized augmented Lagrangian (PAMAL) method \cite{Chen}. SOC method is quite simple to implement and, even if no convergence proof has been provided yet, extensive numerical results validate the effectiveness and robustness of such a strategy. Conversely, PAMAL algorithm, which hybridizes the augmented Lagrangian method and the proximal minimization scheme, is known to guarantee convergence. Furthermore, any limit point of each sequence generated by PAMAL method satisfies the Karush-Kuhn Tucker conditions of the original non-convex problem \cite{Chen}.
Finally, to prevent the resulting basis vectors to be excessively sparse vectors, we consider the minimization of a continuous relaxation of  the  balanced cut size. To solve the corresponding non-convex fractional problem, we adopt an efficient and convergent algorithm based on the explicit-implicit gradient method \cite{Bresson2012}.

The paper is organized as follows. Sec. \ref{section:Min-cut size}  introduces the graph signal variations as the continuous
Lov\'{a}sz extension of the min-cut size. In Sec. \ref{section:Fourier_basis}, we define the GFT as the set of optimal orthonormal vectors minimizing the graph signal variation, and in Sec. \ref{section:optm_methods} we illustrate the optimization methods used for solving the resulting non-convex problem. Therefore, in  Sec. \ref{section:balanced} we conceive the GFT as the solution of a balanced min cut problem, while Sec. \ref{section:numerical_results} illustrates some numerical examples validating the effectiveness of the proposed approaches.
Finally, Sec. \ref{section:conclusions} draws some conclusions.
\vspace{-0.1cm}
\section{Min-cut size and its Lov\'{a}sz extension} \vspace{-0.1cm}
\label{section:Min-cut size}
In this section, we recall the definitions of cut size and Lov\'{a}sz extension, as they will form the basic tools
for our definition of GFT. We consider a graph $\mathcal{G}=\{\mathcal{V},\mathcal{E}\}$ consisting of  a set of $N$ vertices (or nodes) $\mathcal{V}=\{1,\ldots, N\}$ along with a set of edges $\mathcal{E}=\{a_{ij}\}_{i,j \in \mathcal{V}}$, such that $a_{ij}>0$ if there is a direct link from node $j$ to node $i$, or $a_{ij}=0$ otherwise. We denote with $|\mathcal{V}|$
the cardinality of $\mathcal{V}$, i.e. the number of elements of $\mathcal{V}$.
A signal $\bs$ on a graph $\mathcal{G}$ is defined as a mapping from the vertex set
to a real vector of size $N=|\mathcal{V}|$, i.e. $\bs: \mathcal{V}\rightarrow \mathbb{R}$.
 Let $\mA$ denote the $N\times N$ adjacency matrix with entries given by the edge weights $a_{ij}$ for $i,j=1,\ldots,N$.
The  graph Laplacian is defined as $\mL:= \mD-\mA$ where the in-degree matrix $\mD$ is a diagonal matrix
whose $i$th diagonal entry is $d_i= \sum_j a_{ij}$.

One of the basic operations over graphs is clustering, i.e. the partition of the graph onto disjoint subgraphs, such that the vertices within each subgraph (cluster) are highly interconnected, whereas there are only a few links between different clusters. Finding a good partition can be formulated as the minimization of the cut size \cite{Newman}, whose definition is reported here below. Let us consider a subset of vertices $\mathcal{S}\subset \mathcal{V}$, and its complement set in $\mathcal{V}$ denoted by $\bar{\mathcal{S}}$. The edge boundary of $\mathcal{S}$ is defined as the set of edges with one end in $\mathcal{S}$ and the other end in $\bar{\mathcal{S}}$. The cut size between $\mathcal{S}$ and  $\bar{\mathcal{S}}$ is defined as the sum of the weights over the boundary \cite{Newman}, i.e.
\beq
\label{cut_size}
\text{cut}(\mathcal{S},\bar{\mathcal{S}}):=\sum_{i \in \mathcal{S}, j \in \bar{\mathcal{S}}} a_{ji}.
\eeq
Finding the partition that minimizes the cut size in (\ref{cut_size}) is an NP-hard problem. To overcome this difficulty, we exploit the sub-modularity property of the cut size \cite{Bach2013}, which ensures that its Lov\'{a}sz extension is a convex function  \cite{Bach2013}. We briefly recall some of the main definitions and properties here below. Given the set $\mathcal{V}$ and its power set $2^{\mathcal{V}}$, i.e. the set of all its subsets, let us consider a real-valued set function $F: 2^{\mathcal{V}}\rightarrow \mathbb{R}$.  The cut size in (\ref{cut_size}) is an example of set function, with $F(\mathcal{S}):=\text{cut}(\mathcal{S},\bar{\mathcal{S}})$.
Every element of the power set $2^{\mathcal{V}}$ may be associated to a vertex of the hyper-cube $\{0,1\}^N$. Namely, a set $\mathcal{S} \subseteq {\mathcal{V}}$ can be uniquely identified to the indicator vector $\mathbf{1}_{\mathcal{S}}$, i.e. the vector which
is $1$ at entry $j$, if $j\in \mathcal{S}$, and $0$ otherwise. Then, a set-function $F$
can be defined on the vertices of the hyper-cube $\{0,1\}^N$. The Lov\'{a}sz extension of a graph function $F$ \cite{Lovasz1983}, \cite{Bach2013}, allows the extension of a set-function defined on the vertices of the hyper-cube $\{0,1\}^N$, to the full hypercube $[0,1]^N$ and hence to the entire space $\mathbb{R}^N$. We recall its definition hereafter.
\begin{definition}
Let $F: 2^{\mathcal{V}} \rightarrow \mathbb{R}$ be a set function with $F(\emptyset)=0$.
Let $\bx \in \mathbb{R}^{N}$ be ordered w.l.o.g. in increasing order such that $x_1\leq x_2 \leq \ldots \leq x_N$. Define
$C_0 \triangleq \mathcal{V}$  and $C_i \triangleq \{j \in \mathcal{V} : x_j>x_i\}$ for $i>0$.
Then, the Lov\'{a}sz extension $f: \mathbb{R}^{N} \rightarrow \mathbb{R}$
of $F$, evaluated at $\bx$, is given by:
\beq \label{Lov_def}
\begin{split}
f(\bx)\,=&\,\ds \sum_{i=1}^{N} x_i(F(C_{i-1})-F(C_{i}))\\
\,=& \,\ds \sum_{i=1}^{N-1} F(C_{i})(x_{i+1}-x_i) +x_1  F(\mathcal{V}).
\end{split}
\eeq
\end{definition}
Note that $f(\bx)$ is piecewise affine  w.r.t. $\bx$, and $F(\mathcal{S})=f(\mathbf{1}_{\mathcal{S}})$ for all $\mathcal{S}\subseteq \mathcal{V}$. An interesting class of set functions is given by the submodular set functions, whose definition follows next.

\begin{definition}
A set function  $F : 2^{\mathcal{V}}\rightarrow \mathbb{R}$  is  submodular if and only if, $\forall \mathcal{A},\mathcal{B}\subseteq \mathcal{V}$, it satisfies the following inequality:
\beq \nonumber
F(\mathcal{A})+F(\mathcal{B})\geq F(\mathcal{A} \cup \mathcal{B})+F(\mathcal{A} \cap \mathcal{B}).
\eeq
\end{definition}
A fundamental property of a submodular set function is that its Lov\'{a}sz extension is a convex function. This is formally stated in the following proposition \cite[p.23]{Bach2013}.

\begin{proposition} \label{min_f}
Let  $F : 2^{\mathcal{V}}\rightarrow \mathcal{R}$  be a {\it submodular} function and $f$ be its
Lov\'{a}sz extension. Then, it holds $$\min_{\mathcal{S}\subseteq \mathcal{V}} \; F(\mathcal{S})=\min_{\mx \in \{0,1\}^N} \; f(\mx)=\min_{\mx \in [0,1]^N} \; f(\mx).$$
Moreover, the set of minimizers of $f(\mx)$ on $[0,1]^N$ is the convex hull of the minimizers of $f(\mx)$ on $\{0,1\}^N$.
\end{proposition}

The cut size function  in (\ref{cut_size}) is known for being submodular, see, e.g.,  \cite{Bach2013}, \cite{Hein2013}.
More specifically, as shown  in \cite[p.54]{Bach2013}, the cut function is equal to the positive linear combination of the function $G_{ij} \, : \, \mathcal{S} \mapsto
(\mathbf{1}_{\mathcal{S}})_i[1-(\mathbf{1}_{\mathcal{S}})_j]$, i.e. $$\text{cut}(\mathcal{S})=\ds \sum_{i,j \in \mathcal{V}} a_{ji} G_{ij}.$$
The function $G_{ij}$ is the extension to $\mathcal{V}$ of a function $\widetilde{G}_{ij}$  defined only on the power set of $\{i,j\}$, where $\widetilde{G}_{ij}(\{i\})=1$ and all other values are zero, so that, from (\ref{Lov_def}), its Lov\'{a}sz extension is $\widetilde{G}_{ij}(x_i,x_j)= [x_i-x_j]_{+}$ with $[y]_+:=\text{max}\{y,0\}$. Therefore the Lov\'{a}sz extension of the cut size function, in the general case of directed graphs, is given by:
\beq
\label{TV-directed}
f(\mx)=\ds \sum_{i,j=1}^{N} a_{ji}[x_i-x_j]_{+}:= \text{GDV}(\mx). \eeq
 We term this function the Graph Directed Variation (GDV), as it captures the edges' directivity.
For undirected graphs, imposing $a_{ij}=a_{ji}$, the Lov\'{a}sz extension of the cut size  boils down to
\beq
\label{TV-undirected}
f(\mx)=\ds \sum_{i,j=1, i>j}^{N} a_{ji}|x_i-x_j|:= \text{GAV}(\mx).
\eeq
Interestingly, this function, which we call Graph Absolute Variation (GAV),  represents the discrete counterpart of the $l_1$ norm total variation, which plays a key role in the classical Fourier Transform of continuous time signals \cite{Mallat}, \cite{Zhu2012}.\\
It is easy to show that the directed variation  $\text{GDV}$  satisfies the following properties:
\begin{enumerate}
  \item [i)] $\text{GDV}(\mx)\geq 0$, $\forall \, \mx \in \mathbb{R}^N$;
  \item [ii)] $\text{GDV}(\mx)=0$, $\forall \, \mx=c \mathbf{1}$ with $c\geq 0$;
  \item [iii)] $\text{GDV}(\alpha \,\mx)=\alpha \, \text{GDV}(\mx)$, $\forall \, \alpha\geq 0$, i.e. it is  positively homogeneous;
 \item [iv)] $\text{GDV}(\mx +\my) \leq \text{GDV}(\mx)+\text{GDV}(\my)$, $\forall \, \mx,\my \in \mathbb{R}^N$.
\end{enumerate}
$\text{GDV}$ is neither a proper norm nor a semi-norm, since, in this latter case, it should be absolutely homogeneous. However, it meets the desired property ii) ensuring  that a constant graph signal has zero total variation.

\section{Graph Fourier Basis and Directed\\ Total Variation}\vspace{-0.1cm}
\label{section:Fourier_basis}
Alternative definitions of GFT have been proposed in the literature, depending on the different perspectives used to emphasize specific signal features. In case of  undirected graphs, the GFT of a vector $\bs$ was defined as \cite{Shuman2013}
\begin{equation}
\label{Ux}
\hat{\bs}=\mU^T \bs,
\end{equation}
where the columns of matrix $\mU$ are the eigenvectors of the Laplacian matrix $\mL$, i.e. $\mL=\mU \mLambda \mU^T$. This definition is basically rooted on the clustering properties of these eigenvectors, see, e.g., \cite{Chung1997}. In fact, by definition of eigenvector, the Fourier basis used in (\ref{Ux}) can be thought as the solution of the following sequence of optimization problems:
\begin{equation}
\label{u_GFT_undir}
\begin{split}
{\bu}_k = \,&\, \underset{\bu_k \in \mathbb{R}^{N}}{{\rm arg}\min} \; \ds \bu_k^T \mL \bu_k \,:=\,
\underset{\bu_k \in \mathbb{R}^{N}}{{\rm arg}\min} \,\, \text{GQV}(\bu_k)\\
& \; \mbox{s.t.} \quad \quad \bu_k^T {\bu}_{\ell} = \delta_{k\ell}, \; \;\; \ell=1,\ldots, k,
\end{split}
\end{equation}
for $k=2,\ldots,N$, where $\delta_{k\ell}$ is the Kronecker delta, and we used the property that the quadratic form built on the Laplacian
is the $\ell_2$-norm, or graph quadratic variation (GQV), i.e. $$\text{GQV}(\mx):=\sum_{i,j=1, j>i}^{N} a_{ji}(x_i-x_j)^2.$$
Thus, the Fourier basis obtained from  (\ref{u_GFT_undir}) coincides with the set of orthonormal vectors that minimize the $\ell_2$-norm total variation. In all applications where the graph signals exhibit a  cluster behavior, meaning that the signal is relatively smooth within each cluster, whereas it can vary arbitrarily from cluster to cluster, the GFT defined as in (\ref{Ux}) helps emphasizing the presence of clusters \cite{Chung1997}.  However, the identification of the Laplacian eigenvectors as the orthonormal vectors that minimize the GQV is only valid for undirected graphs, for which  the quadratic form built on the Laplacian reduces to the GQV. For directed graphs, the quadratic form in (\ref{u_GFT_undir}) captures only properties associated to the symmetrized Laplacian (i.e., $\bL_s=(\bL+\bL^T)/2$), and hence it cannot capture the edges' directivity. The generalization to  {\it directed} graphs,  was proposed in \cite{Moura2014} as
\begin{equation}
\label{V-1x}
\hat{\bs}=\mV^{-1} \bs,
\end{equation}
where $\mV$ comes from the Jordan decomposition of the  nonsymmetric adjacency matrix $\mA$, i.e. $\mA=\mV \mJ \mV^{-1}$.
To estimate variations of the graph Fourier basis and to identify an order among frequencies, the total variation of a vector was defined in \cite{Moura2014} as
\beq
\label{TV_dir}
\text{TV}_A(\bs)=\ds \| \bs - \mA_{\text{norm}}\,\bs \|_1,
\eeq
where $\mA_{\text{norm}}\!:=\!\mA/|\lambda_{\text{max}}(\mA)|$.
The previous definition leads to the elegant theory of algebraic signal processing over graphs \cite{Puschel_1, Puschel_2, Moura2013, Moura2014}. However, there are some critical issues associated to that definition that need to be further explored. First, the definition of total variation as given in (\ref{TV_dir}) does not ensure that a constant graph signal has zero total variation, and this collides with the common meaning of total variation \cite{Mallat}, \cite{Zhu2012}, \cite{Lozes}.
Second, the columns of  $\mV$ are linearly independent complex generalized eigenvectors, but in general they are not orthogonal. This gives rise to a GFT that does not preserve inner products when passing from the observation to the transformed domain. Furthermore,  the computation of the Jordan decomposition incurs into serious and intractable numerical instabilities when the  graph size exceeds even moderate values \cite{Golub76} and  more stable matrix decomposition methods have to be adopted to tackle its instability issues \cite{girault}. To overcome some of these criticalities, very recently the authors of  \cite{Singh2016} proposed a shift operator based on the directed Laplacian of a graph. Using the Jordan decomposition, the graph Laplacian is decomposed as
\beq \label{GFTSingh}
\mL=\mV_L \mJ_L \mV_L^{-1}
\eeq
and the GFT is defined in \cite{Singh2016} as
\begin{equation}
\label{V-L}
\hat{\bs}=\mV^{-1}_{L} \bs.
\end{equation}
To quantify oscillations in the graph harmonics and to order the frequencies, the total variation was defined in  \cite{Singh2016} as
\beq
\label{GTV_dirL}
\text{TV}_L(\bs)=\ds \| \mL \,\bs \|_1.
\eeq
This definition of total variation ensures a zero value for constant graph signals. Furthermore,  the eigenvalues with small absolute value correspond to low frequencies. Nevertheless, the GFT given by $\mF=\mV^{-1}_{L}$ is still a non-unitary transform and its computation is affected by the numerical instabilities associated to the Jordan decomposition.

In this paper, we propose a novel method to build the graph Fourier basis as the set of $N$ orthonormal vectors $\mx_i, i=1, \ldots, N$, that minimizes the total variation defined in (\ref{TV-directed}), which represents the continuous convex Lov\'{a}sz extension of the graph cut size in (\ref{cut_size}). The first vector is certainly the constant vector, i.e. $\mx_1=b \, \mathbf{1}$, with $b=1/\sqrt{N}$, as this (unit-norm) vector yields a total variation equal to zero. Let us introduce the matrix $\mX:=(\mx_1, \ldots, \mx_N)\in \mathbb{R}^{N \times N}$  containing all the basis vectors.
Thus, the search for the GFT basis  can be formally stated as the search for the orthonormal vectors that minimize the directed total variation in (\ref{TV-directed}), i.e.
\beq \nonumber
\begin{split} \label{prob_main}
 & \underset{\mathbf{X}\in \mathbb{R}^{N\times N}}{\min}  \quad \;  \text{GDV}(\mX)\,:=\,\sum_{k=1}^{N} \text{GDV}(\mx_k) \hspace{1.5cm} (\mathcal{P})\medskip\\
 & \hspace{0.4cm}  \begin{array}{cll}   \hspace{-0.2cm} \mbox{s.t.}   &  \quad  \;\; \mX^T \mX=\mI,\quad\mx_1=b\mathbf{1}. \end{array}\end{split}
\eeq
The constraints are used to find an orthonormal basis and to prevent the trivial null solution. Although the objective function is convex, problem $\mathcal{P}$ is non-convex due to the orthogonality constraint. In the next section, we present two alternative optimization strategies aimed at solving the non-convex, non-differentiable problem $\mathcal{P}$ in an efficient manner.

\section{Optimization Algorithms}
\label{section:optm_methods}
To avoid handling the non-convex orthogonality constraints  directly, several methods have been proposed in the literature based on the solution of a sequence of unconstrained problems approaching the feasibility condition, such as the penalty methods \cite{Nocedal}, \cite{Bethuel} and the augmented Lagrangian based methods \cite{Bertsekas}, \cite{Fortin}. The penalty method is generally simple, but it suffers from slow-convergence and ill-conditioning. On the other hand, the  standard augmented Lagrangian method  solves a sequence of  sub-problems that usually have no analytical solutions and the choice of the initial points, ensuring  a fast convergence rate,  is usually nontrivial. To cope with these issues, in this section   we present two alternative iterative algorithms to solve the non-convex, non-smooth problem $\mathcal{P}$,   hinging on some recently developed methods for solving non-differentiable problems with  non-convex constraints \cite{Lai2014},\cite{Chen}.
The first method, introduced in \cite{Lai2014}, called splitting orthogonality constraints (SOC) method,
is based on the alternating method of multipliers (ADMM) \cite{Boyd_ADMM},\cite{Glowinski} and the split Bregman method \cite{Yin},\cite{Osher}. The SOC method leads to some important benefits, as it is simple to implement and the resulting non-convex sub-problem with orthonormal constraint admits a closed form solution. Although no convergence proof of SOC method has been provided yet,  numerical results validate its value and robustness.\\
An alternative optimization method that tackles the non-convex minimization problem $\mathcal{P}$ and guarantees convergence is the
PAMAL algorithm recently developed in \cite{Chen}. The algorithm combines the augmented Lagrangian method with proximal alternating minimization. A convergence proof was provided in \cite{Chen}.  More specifically, this method has the so-called sub-sequence convergence property, i.e. there exists at least one convergent sub-sequence, and any limit point satisfies the
Karush-Kuhn Tucker (KKT) conditions of the original nonconvex problem. Building on these algorithms, in the sequel we introduce two efficient optimization strategies that build the basis for the Graph Fourier Transform, as the solution of problem $\mathcal{P}$.
\vspace{-0.3cm}
\subsection{SOC method}\vspace{-0.1cm}
The SOC algorithm was developed in \cite{Lai2014} and tackles orthogonality constrained problems by iteratively solving a convex problem and a quadratic problem that admits a closed-form solution. More specifically, introducing an auxiliary variable $\mP=\mX$
to split the orthogonality constraint, problem $\mathcal{P}$ is equivalent to \vspace{-0.1cm}
\beq  \label{split}
\begin{split}
 & \underset{\mathbf{X}, \mathbf{P} \in \mathbb{R}^{N \times N}}{\min}  \quad \text{GDV}(\mX) \medskip\\
 & \hspace{0.4cm}  \begin{array}{cll}    \mbox{s.t.}   &  \quad \;\; \mX=\mP,\quad \mx_1=b\mathbf{1}, \quad \mP^T \mP=\mI. \\
 \end{array}\end{split}\vspace{-0.1cm}
\eeq
The first constraint is linear and, as discussed in
\cite{Lai2014}, it can be solved using Bregman iteration. Therefore, by adding the Bregman penalty function \cite{Yin},  problem (\ref{split})
 is equivalent to the following simple two-step procedure:\vspace{-0.1cm}
\beq \nonumber
\begin{split}
 &(\mX^k,\mP^k) \triangleq  \underset{\mathbf{X}, \mathbf{P} \in \mathbb{R}^{N \times N}}{\arg \min}  \quad \text{GDV}(\mX)+\ds\frac{\beta}{2}\| \mX-\mP+\mB^{k-1}\|^{2}_{F} \medskip\\
 &\hspace{2.2cm}  \begin{array}{cll}    \mbox{s.t.}   & \quad \;\;\; \mx_1=b\mathbf{1}, \;\; \mP^T \mP=\mI;\end{array}  \\
 & \mB^k=\mB^{k-1}+\mX^k-\mP^k, \end{split}
\eeq
where $\beta$ is a strictly positive constant. Similarly to ADMM and split Bregman iteration \cite{Goldstein}, the above problem can be solved by iteratively minimizing with respect to $\mX$ and $\mP$:
\beq
\begin{array}{lll}
    \begin{array}{ll}  \text{1.} \quad \mX^k \triangleq & \underset{\mathbf{X} \in \mathbb{R}^{N \times N}}{\arg \min}  \quad \text{GDV}(\mX)+\ds\frac{\beta}{2}\| \mX-\mP^{k-1}+\mB^{k-1}\|^{2}_{F}  \\
     & \hspace{0.4cm} \mbox{s.t.}  \hspace{0.8cm} \mx_1=b\mathbf{1}  \hspace{3cm} (\mathcal{P}_k)\end{array} \medskip\\
    \begin{array}{ll}  \text{2.} \quad  \mP^k \triangleq &\underset{\mathbf{P} \in \mathbb{R}^{N \times N}}{\arg \min}  \quad \| \mP-(\mX^{k}+\mB^{k-1})\|^{2}_{F}\medskip\\
    & \hspace{0.4cm} \mbox{s.t.} \hspace{0.8cm} \mP^T\mP=\mI  \hspace{2.8cm} (\mathcal{Q}_k)\end{array}\medskip\\
    \begin{array}{ll} \text{3.}  \quad  \mB^k=\mB^{k-1}+\mX^k-\mP^k. \end{array}
\end{array} \label{iter_alg}
\eeq
The interesting aspect of this formulation is that subproblem $\mathcal{P}_k$ is convex and the second constrained quadratic problem $\mathcal{Q}_k$ has a closed-form solution, as illustrated in the following proposition.
\begin{proposition}
 Define $\mY^k=\mX^k+\mB^{k-1}$ and let $$\mY^k=\bar{\mQ}\mS \bar{\mR}^T$$ be its SVD decomposition, where $\bar{\mQ}, \bar{\mR} \in  \mathbb{R}^{N\times N}$ are unitary matrices, and $\mS\in \mathbb{R}^{N\times N}$ is the diagonal matrix  with  entries the singular values of $\mY^k$. Then, the optimal solution of the quadratic non-convex problem $\mathcal{Q}_k$ in (\ref{iter_alg}) is  $\mP^{k}=\bar{\mQ} \bar{\mR}^T$.
\end{proposition}

\begin{proof}
See the  proof of  Theorem $2.1$ in \cite{Lai2014}.
\end{proof}

Combining (\ref{iter_alg}) and Proposition 2, the main steps of the SOC method are summarized in Algorithm \ref{algorithm:Alg_SOC}. It is important to remark that the choice of the coefficient $\beta$ strongly affects the convergence behavior of the algorithm: a large
value of $\beta$ will force a stronger equality constraint, while a too small $\beta$ might not be able to guarantee the solution to satisfy the
orthogonality constraint. Then, a proper tuning of the coefficient $\beta$ is important to ensure a fast convergence of the algorithm.
Although, as remarked in \cite{Lai2014}, the convergence analysis of SOC algorithm is still an open problem, we will show next that
the numerical results testify the validity and robustness of this method when applied to our case.

\noindent
 \begin{algorithm}[t]
\small

    \quad  {Set}  $\beta>0$, $\mX^0 \in \mathbb{R}^{N \times N}$, $\mX^{0 \, T} \mX^0=\mI$, $\mx_1=b \mathbf{1}$, $\mP^0=\mX^0$,

     \quad $\mB^0=\mathbf{0}$, $k=1$.

    \quad  {\textbf{Repeat}}

 \quad \quad           {Find} $\mX^k$ {as} {solution} {of} $\mathcal{P}_k$ in (\ref{iter_alg}),

 \quad \quad           $\mY^k=\mX^k+\mB^{k-1}$,

   \quad \quad          {Compute} {SVD} {decomposition} $\mY^k=\bar{\mQ} \mS \bar{\mR}^T$,

\quad \quad           $\mP^k=\bar{\mQ} \bar{\mR}^T$,

\quad \quad           $\mB^k= \mB^{k-1}+\mX^k-\mP^k$,

  \quad \quad         $k=k+1$,

  \quad {\textbf{until}}   {\textbf{convergence}.}

    \caption{\!\!: SOC method}
 \label{algorithm:Alg_SOC}
\end{algorithm}
\vspace{-0.1cm}
\subsection{PAMAL method}
As an alternative efficient method to tackle the non-convexity of problem $\mathcal{P}$, we propose here an approach based on PAMAL algorithm \cite{Chen}. The method solves the orthogonality constrained problem by iteratively  updating the primal variables and the multipliers estimates. To this end, let us reformulate the problem as follows. Let us introduce the sets $\mathcal{S}_{\mathbf{1}}$, defined as
$\mathcal{S}_{\mathbf{1}}\triangleq \{\mx=\pm b \mathbf{1}\}$, and
 $\mathcal{S}_t\triangleq \{\mP \in \mathbb{R}^{N \times N} \; : \;\mP^T \mP=\mI\}$, which represents  the  Stiefel manifold \cite{Manton}.
For any set $\mathcal{S}$, its indicator function is defined as
\beq
\delta_{\mathcal{S}}(\mX)=\left\{ \begin{array}{lll} 0, & \text{if} \; \mX\in \mathcal{S}\\
+\infty, & \text{otherwise}. \end{array} \right.
\eeq
Given these symbols, problem (\ref{split}) is equivalent to the following one:\vspace{-0.02cm}
\beq \nonumber
\begin{split}
 & \underset{\mathbf{X}, \mathbf{P} \in \mathbb{R}^{N \times N}}{\min}  \quad f(\mX,\mP)\triangleq\text{GDV}(\mX)+\delta_{\mathcal{S}_{\mathbf{1}}}(\mx_1)+\delta_{\mathcal{S}_t}(\mP) \hspace{0.5cm} (\mathcal{P}_e)\\
 & \hspace{0.4cm}  \begin{array}{cll}    \mbox{s.t.}   &  \quad \; \; \mathbf{H}(\mX,\mP)\triangleq \mP-\mX=\mathbf{0}. \\
 \end{array}\end{split}
\eeq
The basic idea to solve a problem in the form of $\mathcal{P}_e$ was proposed in \cite{Chen}, and combines the augmented Lagrangian method \cite{Andreani},\cite{Bertsekas} with the alternating proximal minimization algorithm. The result is known as the PAM method \cite{Bolte}, which deals with non-smooth, non-convex optimization. According to the augmented Lagrangian method, we add a penalty term to the objective function in order to associate a high cost to unfeasible points. In particular, the augmented Lagrangian function
associated to the non-smooth problem $\mathcal{P}_e$, is
 \beq \label{Lagrange} \nonumber
 \mathcal{L}(\mX,\mP,\mLambda)= f(\mX,\mP)+\langle \mLambda, \mathbf{H}(\mX,\mP)\rangle+\frac{\rho}{2}\|\mathbf{H}(\mX,\mP) \|^{2}_{F},
 \eeq
where $\rho$ is  a positive penalty coefficient, $\mLambda \in \mathbb{R}^{N \times N}$ represents the multipliers matrix, while the matrix inner product is defined as $\langle \mA, \mB\rangle \triangleq \text{tr}(\mA^T  \mB)$. The proposed augmented Lagrangian method reduces  problem $\mathcal{P}_e$ to a sequence of problems that alternately update, at each iteration $k$,  the following three steps:
 \begin{enumerate}
   \item [1.]
   Compute the critical point $(\mX^k,\mP^k)$
of the function  $\mathcal{L}(\mX,\mP,\mLambda^k; \rho^k)$  by solving
\beq \label{min_lag}
\begin{split}
 (\mX^k,\mP^k) \triangleq & \underset{ \mathbf{X}, \mathbf{P} \in \mathbb{R}^{N \times N}}{\min}  \quad \mathcal{L}(\mX,\mP,\mLambda^k;\rho^k);
 \end{split}
\eeq
   \item [2.]  Update the multiplier estimates $\mLambda^k$;
   \item [3.] Update the penalty parameter $\rho^k$.
 \end{enumerate}
We will show next how to implement the previous steps, which are described in detail in Algorithm \ref{algorithm:Alg_PAMAL_1}.\\
\noindent \textbf{Computation of the critical points $(\mX^k,\mP^k)$}. The optimal solution  $(\mX^k,\mP^k)$ of problem (\ref{min_lag}) is  computed using an approximate algorithm, i.e.  finding a subgradient point  $\mTheta^k \in \partial \mathcal{L}(\mX^k,\mP^k,\mLambda^k; \rho^k)$ satisfying, with a prescribed tolerance value $\epsilon^k$, the following inequality
\beq \label{norm_ineq}
\parallel  \mTheta^k \parallel_{\infty}\leq \epsilon^k
\eeq
 with $\mP^k \in \mathcal{S}_t$. To evaluate such point, we exploit a coordinate-descent method with proximal regularization based on the PAM method proposed in \cite{Attouch}. More specifically, at the $k$-th outer iteration of the algorithm, we compute $(\mX^k,\mP^k)$ by iteratively solving, at each inner iteration $n$, the following proximal regularization of a two blocks Gauss-Seidel method:
\begin{align}
&\mX^{k,n} =\underset{\mathbf{X} \in \mathbb{R}^{N \times N}, \mx_1=b\mathbf{1} }{{\rm arg}\min} \,  \mathcal{L}(\mX,\mP^{k,n-1},\mLambda^k; \rho^k) \nonumber \\
 & \hspace{3.3cm} +\frac{c_1^{k,n-1}}{2} \parallel \mX-\mX^{k,n-1}\parallel^{2}_{F}  \hspace{0.3cm} (\tilde{\mathcal{P}}_{k,n}) \nonumber \\
&\mP^{k,n} = \hspace{0.5cm} \underset{\mathbf{P} \in \mathbb{R}^{N \times N}}{{\rm arg}\min} \, \hspace{0.5cm} \mathcal{L}(\mX^{k,n-1}, \mP,\mLambda^k; \rho^k) \nonumber  \\
&  \hspace{3.3cm} +\frac{c_2^{k,n-1}}{2} \parallel \mP-\mP^{k,n-1}\parallel^{2}_{F} \hspace{0.3cm} (\tilde{\mathcal{Q}}_{k,n}) \nonumber
\end{align}
where the proximal parameters $c_i^{k,n}$ can be arbitrarily chosen as long as they satisfy
\beq \label{eq:c_bounds}
0<\underline{c}\leq c_i^{k,n}\leq \bar{c} < \infty, \; k,n \in \mathbb{N},\; i=1,2, \;\underline{c}>0,\; \bar{c}>0.
\eeq
The first convex problem $\tilde{\mathcal{P}}_{k,n}$ can be solved through any convex optimization numerical tool, whereas the second
problem in  $\tilde{\mathcal{Q}}_{k,n}$ admits a closed-form solution as stated in the following proposition.
 \begin{proposition}
   Define the matrix  $$\mF\triangleq (c_2^{k,n-1} \mP^{k,n-1}+\rho^k \mX^{k,n}-\mLambda^{k})(\rho^k+c_2^{k,n-1})^{-1}$$ with SVD decomposition
 $\mF=\mQ \mSigma \mT^T$,  where $\mQ, \mT \in \mathbb{R}^{N\times N}$ are unitary matrices, while $\mSigma$
  is a diagonal matrix with entries given by the singular values of $\mF$. The optimal solution
 of the non-convex problem $\tilde{\mathcal{Q}}_{k,n}$ is given by $\mP^{k,n}=\mQ \mT^T$.
 \end{proposition}
 \begin{proof}
 See Appendix \ref{A:closed_form}.
 \end{proof}\\
Algorithm \ref{algorithm:Alg_PAMAL_1} describes the outer loop of the PAMAL method whereas in  Algorithm \ref{algorithm:Alg_PAMAL_2} we report the inner iterations needed to solve problems $\tilde{\mathcal{P}}_{k,n}$ and $\tilde{\mathcal{Q}}_{k,n}$ in step $1$ of Algorithm \ref{algorithm:Alg_PAMAL_1}.\\
The inner iterations are terminated when there exists a subgradient point $\mTheta^{k,n} \in \partial \mathcal{L}(\mX^{k,n},\mP^{k,n},\mLambda^k; \rho^k)$ satisfying
     $\parallel \mTheta^{k,n} \parallel_{\infty} \leq \epsilon^k$, $\mP^{k,n} \in \mathcal{S}_t$, where
     $\mTheta^{k,n}\triangleq(\mTheta^{k,n}_1,\mTheta^{k,n}_2)$ with the subgradients given by
\beq \label{thetak}
\begin{array}{lll}
\mTheta^{k,n}_1=c_1^{k,n-1}(\mX^{k,n-1}-\mX^{k,n})+\rho^{k}(\mP^{k,n-1}-\mP^{k,n})\smallskip\\
\mTheta^{k,n}_2=c_2^{k,n-1}(\mP^{k,n-1}-\mP^{k,n}).
\end{array}
\eeq

\noindent \textbf{Update of the multipliers and penalty coefficients.}
The rule for updating the multipliers matrix in Step 2 of Algorithm \ref{algorithm:Alg_PAMAL_1}
needs some further discussion.  We adopt the classical first-order approximation by imposing that the estimates of multipliers must be bounded. Then, we explicitly  project the multipliers matrix on the  compact box  set
 ${\mathcal{T}}\! \triangleq \{\mLambda \!  :\! \mLambda_{min} \!\leq \! \mLambda \leq \mLambda_{max} \}$
 with   $-\infty<[\mLambda_{min}]_{i,j}\leq [\mLambda_{max}]_{i,j}<\infty$, $\forall i,j$.
The boundedness of the multipliers is a fundamental assumption needed to preserve the property that global minimizers of the original problem are obtained if each outer iteration of the penalty method computes a global minimum of the subproblem. Unfortunately, assumptions that imply boundedness of multipliers tend to be very strong and often hard to be verified. Nevertheless, following \cite{Chen}, \cite{Andreani}, \cite{Birgin}, we also impose the boundedness of the multipliers. This implies that, in the convergence proofs, we will assume that the true multipliers fall within the bounds imposed by the algorithm, see, e.g. \cite{Chen}. Regarding the setting of the remaining parameters of the proposed algorithm, we will assume that: i) the sequence of positive tolerance parameters $\{\epsilon^k\}_{k\in \mathbb{N}}$ is chosen such that $\lim_{k\rightarrow \infty} \epsilon^k=0$; ii) the penalty parameter $\rho^k$ is updated according to the infeasibility degree by following the rule described in step 3 of  Algorithm \ref{algorithm:Alg_PAMAL_1} \cite{Chen}, \cite{Bertsekas}.

 \begin{algorithm}[t]
\small
 {Given} {the} {parameters} $\{\epsilon^{k}\}_{k \in \mathbb{N}}$, $0<\epsilon^{k}<1$, $\tau \in [0,1)$, $\gamma>1$, $k=1$, $\rho^k >0$,
$\mLambda^{k} \in \mathbb{R}^{N \times N}$, $\mLambda_{min}\leq \mLambda^{k} \leq \mLambda_{max}$.

 {\textbf{Repeat}}

{\textbf{Step.1:}} {Compute} $(\mX^k,\mP^k)$ as in Algorithm \ref{algorithm:Alg_PAMAL_2} such that there exists

    \quad          $\mTheta^k \in \partial \mathcal{L}(\mX^k,\mP^k,\mLambda^k; \rho^k)$ with
     $\parallel \mTheta^k \parallel_{\infty} \leq \epsilon^k$, $(\mP^{k})^T \mP^k=\mI$.

   {\textbf{Step.2:}} Update the multiplier estimates \smallskip

   \quad \quad  \quad \quad \quad \quad $\mLambda^{k+1}=[\mLambda^k+\rho^k(\mP^k-\mX^k)]_{\mathcal{T}}$\smallskip

    \quad \quad \quad    where $[\cdot]_{\mathcal{T}}$ is the projection on ${\mathcal{T}}\! \triangleq \! \{\mLambda \!  :\! \mLambda_{min} \!\leq \! \mLambda \leq \mLambda_{max} \}$.

   {\textbf{Step.3:}} Set $\mR^k=\mP^k-\mX^k$, and update the penalty parameter as\smallskip

   \quad \quad  \quad \quad $\rho^{k+1}=\left\{ \begin{array}{lll} \rho^k  \quad \; \text{if} \quad \parallel \mR^k \parallel_{\infty} \leq \tau \parallel \mR^{k-1} \parallel_{\infty} \\ \gamma \rho^k \; \; \text{otherwise} \end{array},\right.$

  \quad \quad  \quad \quad      $k=k+1$,

{\textbf{until}}   {\textbf{convergence}.}

       \caption{\!\!: PAMAL method}
 \label{algorithm:Alg_PAMAL_1}
\end{algorithm}

\noindent \textbf{Convergence Analysis.} We now discuss in details the convergence properties of the proposed PAMAL method.
Assume that: i) the proximal parameters $\{c_{i}^{k,n}\}_{\forall k,n}$ are arbitrarily chosen as long as they satisfy (\ref{eq:c_bounds}); ii) the sequence $\{\epsilon^k\}_{k\in \mathbb{N}}$ is chosen such that $\lim_{k\rightarrow \infty} \epsilon^k=0$; iii) the penalty parameter $\rho^k$ is updated according to the rule described in Algorithm \ref{algorithm:Alg_PAMAL_1}.
The PAM method, as given in Algorithm \ref{algorithm:Alg_PAMAL_2}, guarantees global convergence to a critical point \cite[Th. 6.2]{Attouch}, provided that the penalty parameters $\{\rho^k\}_{k\in \mathbb{N}}$ in Algorithm \ref{algorithm:Alg_PAMAL_1} satisfy some mild conditions, as stated in the following theorem.
\begin{theorem}\label{thm:Th1}
 Denote by $\{(\mX^{k,n},\mP^{k,n})\}_{n \in \mathbb{N}}$ the sequence generated by Algorithm 3.
 The function $\mathcal{L}_k$ in (\ref{Lagrange}) satisfies the Kurdyka-{\L}ojasiewicz (K-{\L}) property\footnote{The reader can refer to    Appendix \ref{B:proof Th1} for a definition of the Kurdyka-{\L}ojasiewicz (K-{\L}) property.}.
 Then $\mTheta^{k,n}$ defined by (\ref{thetak}) satisfies
 \beq
 \mTheta^{k,n} \in \partial \mathcal{L}(\mX^{k,n},\mP^{k,n}, \mLambda^{k}; \rho^k), \quad \forall n \in \mathbb{N}.
 \eeq
Also, if $\gamma>1$, $\rho^1>0$, for each $k \in \mathbb{N}$, it holds
\beq
\parallel \mTheta^{k,n} \parallel_{\infty}\rightarrow 0, \;\; \text{as} \; n\rightarrow \infty.
\eeq
\end{theorem}
\begin{proof}
See Appendix \ref{B:proof Th1}.
\end{proof}

The convergence claim for Algorithm  \ref{algorithm:Alg_PAMAL_1} to a stationary solution of problem $\mathcal{P}_e$ is stated in the following theorem.
\begin{theorem} \label{thm:Th2}
 Let $\{(\mX^k,\mP^k)\}_{k \in \mathbb{N}}$ be the sequence generated by Algorithm 2.  Suppose $\rho^1>0$ and $\gamma>1$. Then, the
set of limit points of $\{(\mX^k,\mP^k)\}_{k \in \mathbb{N}}$ is non-empty, and every limit point satisfies the KKT conditions of the original problem $\mathcal{P}_e$.
\end{theorem}
\begin{proof}
The proof follows similar arguments as in \cite[Th. 3.1-3.5]{Chen}, and thus is omitted due to space limitation.
\end{proof}

\begin{algorithm}[t]
\small
 {Let} $(\mX^{1,0},\mP^{1,0})$ be any finite initialization.  For $k\geq 2$, set $(\mX^{k,0},\mP^{k,0})=(\mX^{k-1},\mP^{k-1})$,  $n=0$.

 {\textbf{Repeat}}

  {\textbf{Step.1:}} Set $n=n+1$. {Compute} $\mX^{k,n}$ by solving  problem $\tilde{\mathcal{P}}_{k,n}$.

   {\textbf{Step.2:}}  $\mP^{k,n}=\mQ \mT^T$ where $\mQ,\mT$ come from the following SVD

   \quad \quad  \quad decomposition

   \quad \quad  \quad \quad \quad \quad $\mQ \mSigma \mT^T=\frac{c_2^{k,n-1} \mP^{k,n-1}+\rho^k \mX^{k,n}-\mLambda^{k}}{\rho^k+c_2^{k,n-1}}$.

   {\textbf{Step.3:}} Set  $(\mX^k,\mP^k)=(\mX^{k,n},\mP^{k,n})$, $\mTheta^k=\mTheta^{k,n}$,

{\textbf{until}}   $\parallel \mTheta^{k,n} \parallel_{\infty} \leq \epsilon^k$.

       \caption{\!\!: PAM method for solving step $1$ in Algorithm \ref{algorithm:Alg_PAMAL_1}}
 \label{algorithm:Alg_PAMAL_2}
\end{algorithm}

\noindent \textbf{Remark $1$}. Note that both Algorithms $1$ and $3$ at each step of their loops have to compute
the SVD of an $N \times N$ matrix. Therefore, at each iteration their computational cost is proportional to $\mathcal{O}(N^3)$. So, clearly, there is a complexity issue that deserves further investigations to enable the application to large size graphs.
In this paper, we have not investigated methods to reduce the complexity of the approach exploiting, for instance, the sparsity of the graphs under analysis. Also, we have not optimized the selection of the parameters involved in both SOC and PAMAL methods. However, even if complexity is an issue, the proposed approach is more numerically stable than the only method available today for the analysis of directed graphs, based on the Jordan decomposition.

\noindent \textbf{Remark $2$}.
The two alternative methods proposed above to solve the non-convex problem $\mathcal{P}$ are robust to random initializations, as testified also by the numerical results presented in the sequel. In terms of implementation complexity, SOC algorithm is easier to code even though,
  to the best of our knowledge, a theoretical proof of its convergence is still lacking.
\vspace{-0.1cm}
\section{Minimization of balanced total variation}\vspace{-0.01cm}
\label{section:balanced}
The minimization of the total variation as in (\ref{prob_main}) is inspired by the min-cut problem. However, in some cases, this might favor the appearance of very sparse vectors or of very small clusters, possibly also isolated nodes. One way to prevent these undesired solutions  passes through the introduction of the {\it balanced cut} \cite{Malik2000}, \cite{Hein2011}.
A popular definition for the balanced cut of undirected graph is the Cheeger cut \cite{Cheeger_70}, which is given by:
\beq \label{bal_cheeger}
\min_{\mathcal{S}\subseteq \mathcal{V}} \; \ds \frac{\text{cut}(\mathcal{S},\bar{\mathcal{S}})}{\min(|\mathcal{S}|, |\bar{\mathcal{S}}|)}.
\eeq
Note that $\min(|\mathcal{S}|,|\bar{\mathcal{S}}|)$ attains its maximum when  $|\mathcal{S}|= |\bar{\mathcal{S}}|=N/2$, so that,
for a given value of $\text{cut}(\mathcal{S},\bar{\mathcal{S}})$, the minimum occurs when $\mathcal{S}$ and $\bar{\mathcal{S}}$ have approximately equal size.
While the problem stated above is NP-hard, a
tight continuous relaxation of the balanced cut problems has recently been shown to provide excellent clustering results
\cite{Hein2011,Hein2010,Bresson2010}. In \cite{Bresson2010},\cite{Bresson2012}   it was proved that the balanced
Cheeger cut problem in (\ref{bal_cheeger}) for undirected graphs admits the following \emph{exact} continuous relaxation
\beq \label{exact_rel}
\min_{\mx \in \mathbb{R}^N} \ds \frac{\sum_i \sum_{j, i>j} a_{j i} \mid x_i-x_j\mid}{\sum_i \mid x_i-\text{m}(\mx) \mid}
\eeq
where $\text{m}(\mx)$ stands for the median value of $\mx$.
Note  that since it holds $\sum_i \mid x_i-\text{m}(\mx) \mid=0$,   $ \forall \mx \in \text{span}\{\mathbf{1}\}$,  problem  (\ref{exact_rel}) is well-defined if $\mx \perp \mathbf{1}$. Then, the problem in (\ref{exact_rel}) can be recast as:
\beq \label{exact_rel1}
\min_{\mx \in \mathbb{R}^N, \mx \perp \mathbf{1}} \ds \frac{\sum_i \sum_{j, i>j} a_{j i} \mid x_i-x_j\mid}{\sum_i \mid x_i-\text{m}(\mx) \mid}.
\eeq
\begin{algorithm}[t]
\small
\quad {\textbf{For}} $k=2,\ldots,N$

\quad \quad \text{Set} $n=0$, $\mx_k^n=\mx^0$  nonzero vector with $\text{m}(\mx_k^n)=0$,

\quad \quad  $\alpha>0$, $0< \epsilon\ll 1$.

\quad \quad {\textbf{Repeat}}

\quad \quad \quad   $\bw^n \in \text{sign}(\mx_k^n)$,

\quad \quad \quad          $\bv^{n}=\bw^n-\text{mean}(\bw^n)\mathbf{1}$,

\quad \quad \quad     $\bh^n=\mx_k^n+\alpha \bv^n$,

\quad \quad \quad       $\hat{\mx}_k^{n+1}=\underset{\mx_k \in \mathcal{X}^b_k}{\arg \min} \;f(\mx_k)+\ds \frac{\text{E}(\mx_k^n)}{2 \alpha}{\parallel \mx_k-\bh^n  \parallel}_2^2 $,

\quad \quad \quad          $\my_k^{n+1}=\hat{\mx}_k^{n+1}-\text{m}(\hat{\mx}_k^{n+1})$,

\quad \quad \quad          $\mx_k^{n+1}=\ds \frac{{\my}_k^{n+1}}{{\parallel {\my}_k^{n+1}  \parallel}_2 }$, $n=n+1$,

    \quad \quad {\textbf{until}} $\mid\text{E}(\mx_k^n)-\text{E}(\mx_k^{n-1})\mid <\epsilon$,

\quad \quad        $\mx_k^{n+1}=\ds \frac{\hat{\mx}_k^{n+1}}{{\parallel \hat{\mx}_k^{n+1}  \parallel}_2 }$,

\quad {\textbf{end}.}
       \caption{: Balanced graph signal variation}
 \label{algorithm:Alg_balanced}

\end{algorithm}\vspace{-0.02cm}
In \cite{Bresson2010} it was proved that (\ref{exact_rel}) is an exact relaxation of the Cheeger cut problem and, for any minimizer
$\mx$, there is a number $\nu$ such that, $\forall i$, the binary  solution ${x}_{\nu}(i)=1$ if ${x}(i)>\nu$ and ${x}_{\nu}(i)=0$ for ${x}(i)\leq \nu$, is also a minimizer of the Cheeger cut problem. Then, from the equivalence of problems  (\ref{exact_rel}) and (\ref{exact_rel1}), this result holds true also for any minimizer of (\ref{exact_rel1}).
In the sequel, we formulate the problem of finding the Fourier basis minimizing the balanced total variation in both cases
of directed and undirected graphs. To this end, let us define the  function
\beq
\text{E}(\mx_k)\triangleq \frac{ f(\mx_k)}{\sum_i \mid x_k(i)-\text{m}(\mx_k) \mid}
\eeq
where  $f(\mx_k)=\text{GAV}(\mx_k)$ in (\ref{TV-undirected}), or $f(\mx_k)=\text{GDV}(\mx_k)$ in (\ref{TV-directed}), in case of  undirected or directed graphs, respectively. According to problem (\ref{exact_rel}),  we can  find  a set of  Fourier bases $\{\mx_k\}_{k=1}^{N}$, with $\mx_1=b \mathbf{1}$, by iteratively solving, for
$k=2,\ldots,N$, the following problem
\beq
\begin{split}
\underset{\mx_k\in \mathbb{R}^N}{\min}& \hspace{0.5cm }\ds \text{E}(\mx_k) \hspace{2.3cm} (\mathcal{P}^b_k)\\
\mbox{s.t.}\hspace{0.3cm } &  \begin{array}{lll}
 \quad {\mx_k}^T {{\mx}_{\ell}}=\delta_{k,\ell}, \;\;  \ell=1,\ldots,k.  \end{array}
\end{split}
\eeq
\begin{figure*}[t!]
   \centering
    \begin{subfigure}[b]{0.28\textwidth}
    \centering \hspace{-0.6cm}
        \includegraphics[width=2.2in, height=1.66in]{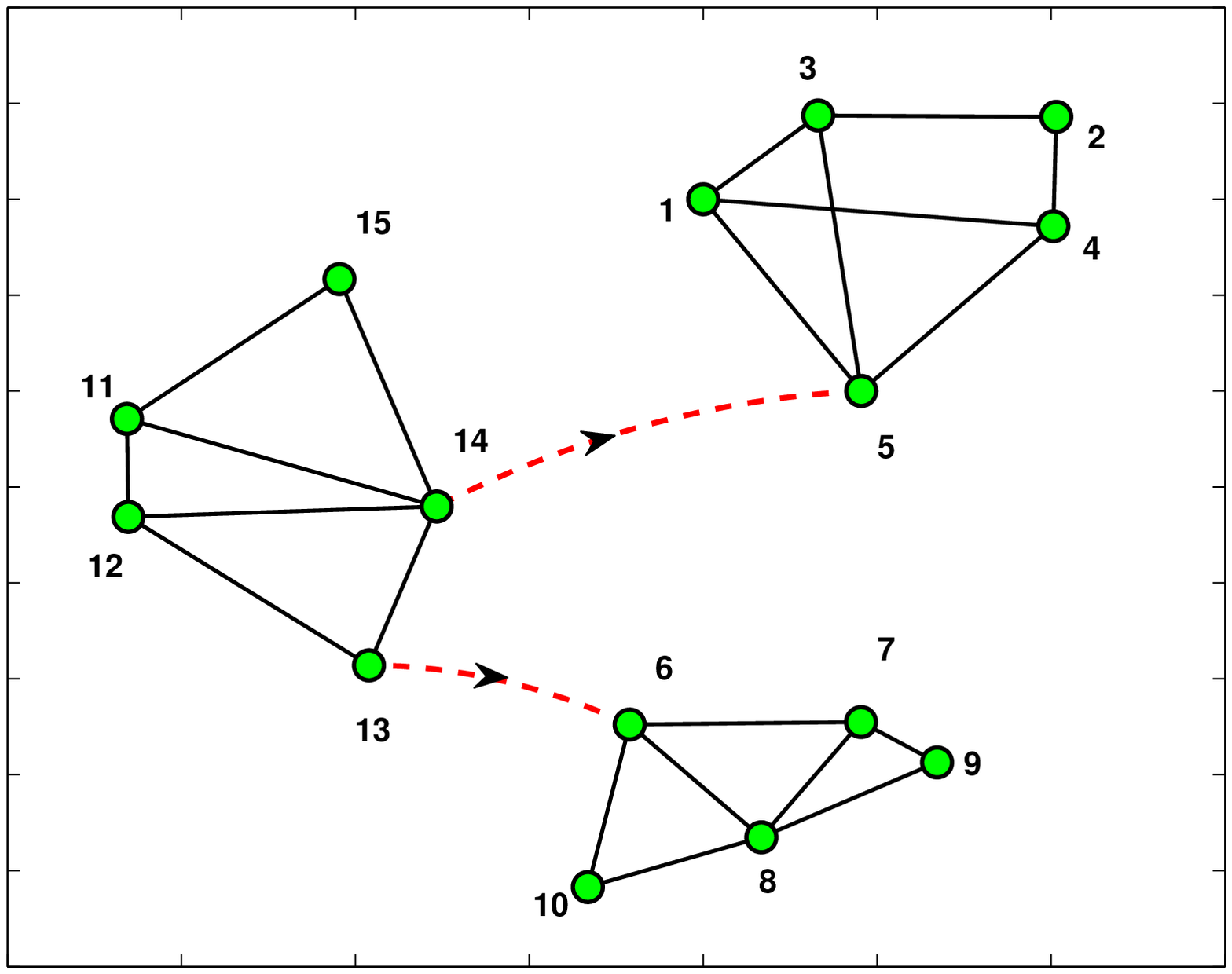}
        \caption{}\label{fig:graphs}
            \end{subfigure}%
            \qquad
              \begin{subfigure}[b]{0.28\textwidth}
        \includegraphics[width=\textwidth]{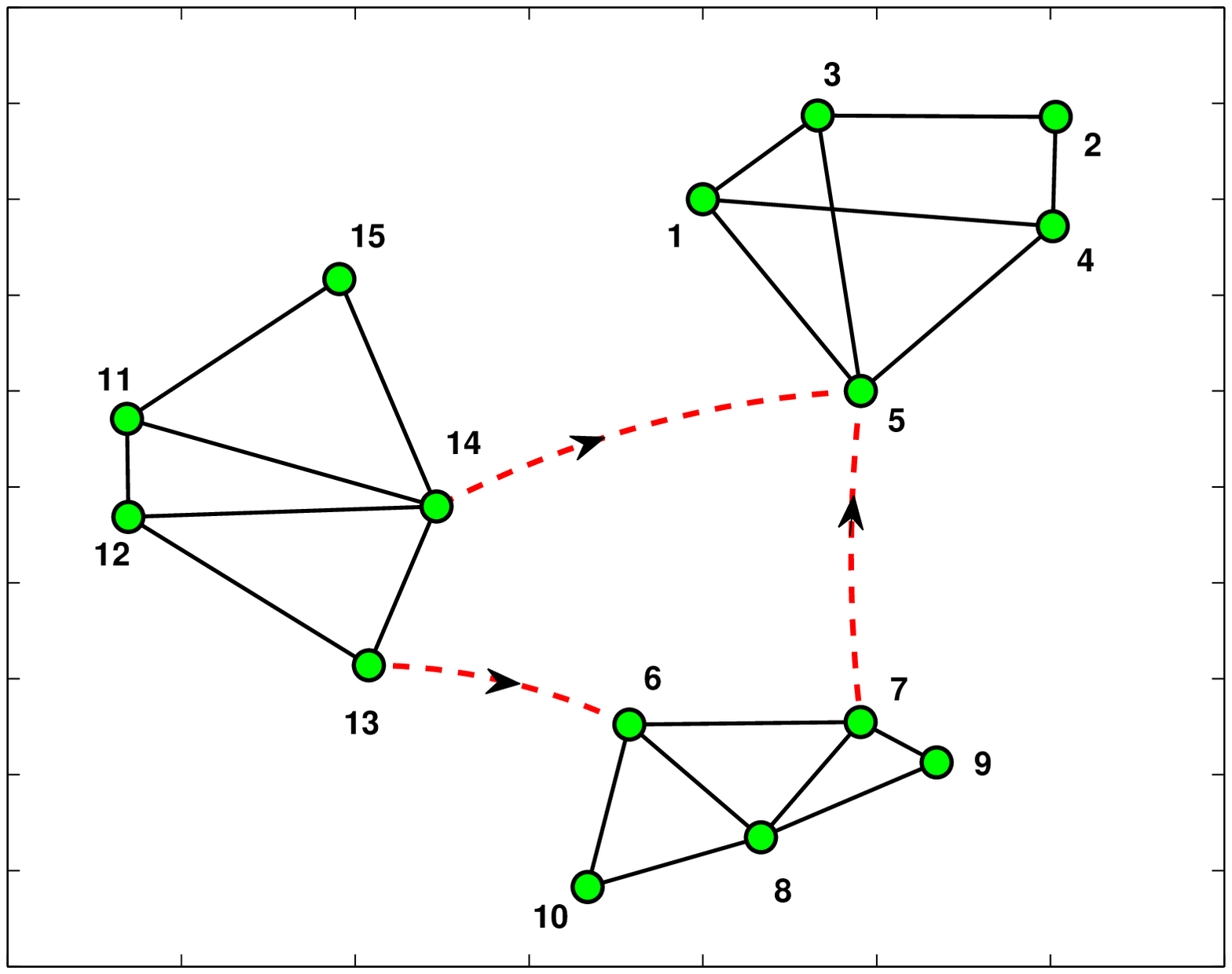}
          \caption{}\label{fig:graphs_3}
           \end{subfigure}%
           \qquad
            \begin{subfigure}[b]{0.28\textwidth}
        \includegraphics[width=\textwidth]{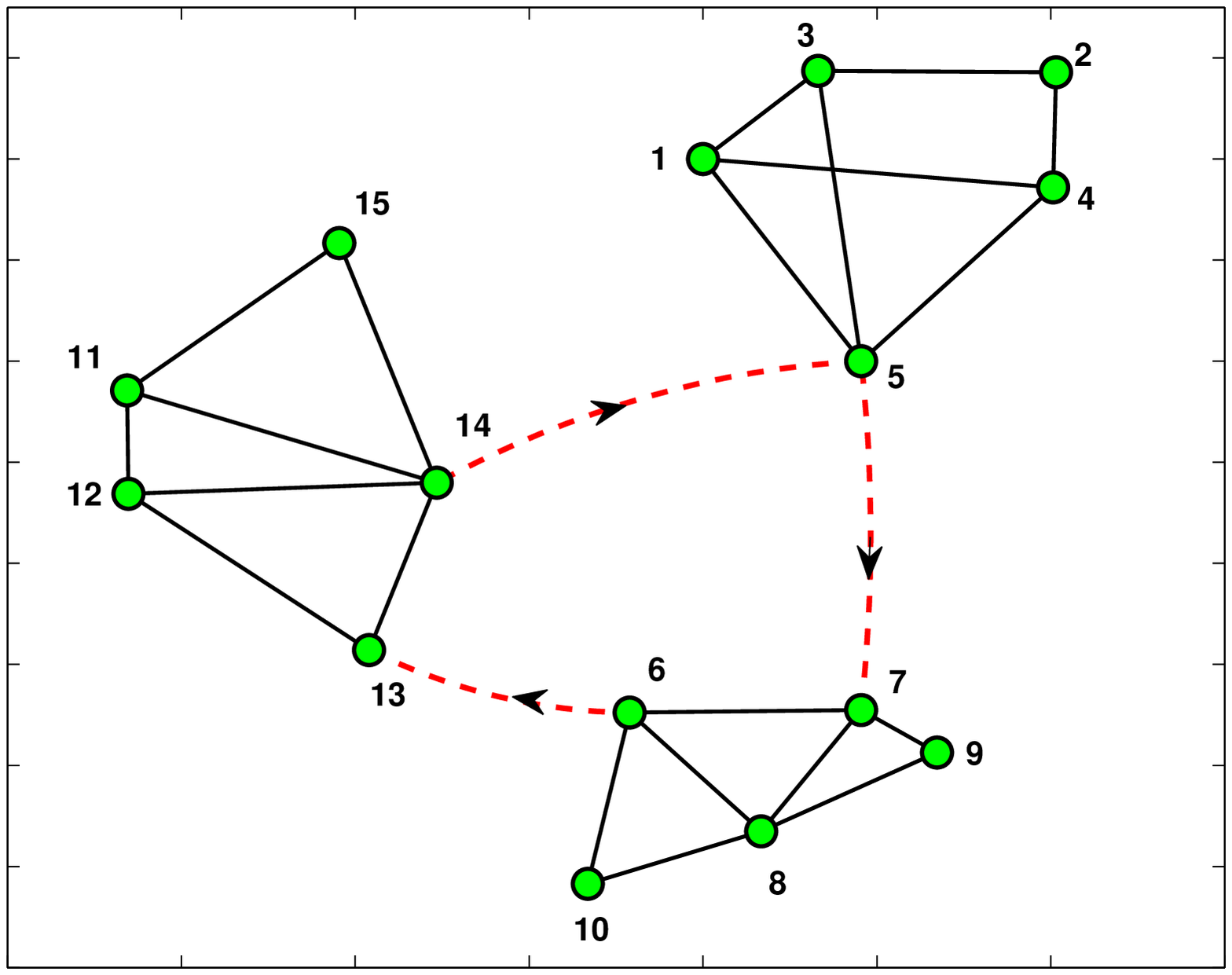}
        \caption{}\label{fig:graphcycle}
           \end{subfigure}%
    \caption{ Examples of graphs with: $(\textit{a})$  $2$ directed links; $(\textit{b})$  $3$ directed links; $(\textit{c})$  $1$ directed cycle.}\label{fig:allgraphs}
   \end{figure*}
Note that problem $\mathcal{P}^b_k$ is non-convex in both the constraints set and the objective function.
Recently, several algorithms \cite{Malik2000}, \cite{Bresson2010}, \cite{Hein2010}, \cite{Hein2011}, have been proposed to minimize relaxations of the balanced cut problem that are similar to (\ref{exact_rel}). Typically, these algorithms give excellent numerical performance, although theoretical convergence proofs are not available.  For instance, in \cite{Bresson2012}, the authors proposed an algorithm minimizing (\ref{exact_rel}), along with a  proof of convergence to a critical point of the original problem. This method is a new steepest descent algorithm based on the explicit-implicit gradient \cite{Boyd} of the function  $\text{E}(\mx_k)\triangleq \frac{ f(\mx_k)}{\text{B}(\mx_k)}$  where  $\text{B}(\mx_k)=\sum_i \mid x_k(i)-\text{m}(\mx_k) \mid$.  The explicit-implicit subgradient of the non-smooth function $\text{E}(\mx_k)$ is given by \vspace{-0.1cm}
\beq \label{eq_sub}
\frac{\mx_{k}^{n+1}-\mx_{k}^{n}}{\tau^n}=-\frac{\partial_{\mx_{k}} f(\mx_{k}^{n+1})-\text{E}(\mx_{k}^{n}) \partial_{\mx_{k}} \text{B}(\mx_{k}^{n})}{\text{B}(\mx_{k}^{n})}\vspace{-0.1cm}
\eeq
or
\beq \label{eq: x_n+1}
\mx_{k}^{n+1}=\mx_{k}^{n}-\tau^n  \frac{\partial_{\mx_{k}} f(\mx_{k}^{n+1})}{\text{B}(\mx_{k}^{n})}+ \tau^n \frac{ \text{E}(\mx_{k}^{n})}{\text{B}(\mx_{k}^{n})}\partial_{\mx_{k}} \text{B}(\mx_{k}^{n}).
\eeq
Let us now consider  the following proximal minimization problem
\beq \label{eq:proximal}
\mx_{k}^{n+1}\triangleq \underset{\mx_{k} \in \mathbb{R}^N}{\arg \min}  \; f(\mx_{k})+\ds \frac{\text{B}(\mx_{k}^{n})}{2 \tau^n}\| \mx_{k}- \mg^{n}\|^2\,.
\eeq
 Any stationary solution of (\ref{eq:proximal}) will be  also  solution of the subgradient equation
\beq
\ds \frac{\tau^n}{\text{B}(\mx_{k}^{n})}\partial_{\mx_{k}} f(\mx_{k})+ \mx_{k}- \mg^{n}=0,
\eeq
so that at step $n+1$ one gets
\beq  \label{eq_xn+1}
\mx_{k}^{n+1}=\mg^{n}-\ds \frac{\tau^n}{\text{B}(\mx_{k}^{n})} \partial_{\mx_{k}} f(\mx_{k}).
\eeq
Replacing  in this last equality the  expression of $\mx_{k}^{n+1}$ given in (\ref{eq: x_n+1}), we  obtain
the following  set of two equations to be iteratively updated:
\beq \nonumber
\begin{array}{lll}
 &\mg^{n}= \mx_{k}^{n}+\tau^{n} \ds\frac{\text{E}(\mx_{k}^{n})}{\text{B}(\mx_{k}^{n})}\bw^{n} \quad \text{with} \quad \bw^{n} \in \partial_{\mx_{k}} \text{B}(\mx_{k}^{n}) \medskip \\
  &\mx_{k}^{n+1}= \underset{\mx_{k} \in \mathcal{X}^b_k}{\arg \min}  \; f(\mx_{k})+\ds \frac{\text{B}(\mx_{k}^{n})}{2 \tau^n}\| \mx_{k}- \mg^{n}\|^2 \label{eq:proximal1}
\end{array}
\eeq
where we define $\mathcal{X}^b_k\triangleq \{ \mx_k\in \mathbb{R}^N \; : \; {\mx_k}^T {{\mx}_{\ell}}=0, \;\text{for} \;  \ell=1,\ldots, k-1\}$. Note that $\mathcal{X}^b_k$ is a set of linear constraints since, for each vector $\mx_k$, the previously computed vectors ${\mx}_{\ell}$, for $\ell=1,\ldots, k-1$, are assumed to be known. The norm one constraint is satisfied through a simple projection of the
optimal solution on the unitary sphere.
 As shown in \cite{Bresson2010}, \cite{Bresson2012}, the algorithm decreases the objective function and preserves the zero mean properties of the successive iterates. It was also observed in \cite{Bresson2012} that a faster convergence rate can be achieved when the step size is chosen as $\tau^{n}= \alpha \frac{\text{B}(\mx_k^{n})}{\text{E}(\mx_k^{n})}$ with $\alpha>0$.

The formal description of the iterative optimization method is given in Algorithm \ref{algorithm:Alg_balanced}, where we denote by $\text{sign}(\ma)$ and $\text{mean}(\ma)$, respectively, the element-wise sign   and the mean value of a vector $\ma$.
The convergence analysis of the algorithm to a critical point of $\text{E}$ was derived in \cite{Bresson2010},\cite{Bresson2012} for undirected graphs. However, since for directed graphs $f(\mx_k)$ preserves all the required properties (i.e., it is non-smooth and convex), the convergence results in \cite{Bresson2010},\cite{Bresson2012}  hold also for the minimization of the balanced directed variation.\vspace{-0.1cm}
\section{Numerical results}\vspace{-0.01cm}
\label{section:numerical_results}
In this section, we present   some numerical results to  assess the effectiveness of the proposed strategy for building the GFT basis.
First, we illustrate some examples of application and then we compare the proposed approach with alternative definitions of GFT basis,
as given in \cite{Shuman2013}, \cite{Moura2014}, \cite{Singh2016}.
In all our experiments, the parameters of SOC and PAMAL methods are set  as (unless stated otherwise): $\beta=100$, $\tau=0.5$, $\gamma=1.5$,
$\rho^1=50$, $\epsilon^k=(0.9)^k, \forall k\in \mathbb{N}$, $\mLambda_{\text{min}}=-1000 \cdot \mI$
$\mLambda_{\text{max}}=1000 \cdot \mI$, $\mLambda^{1}=\mathbf{0}$, $\underline{c}=c_i^{k,n}=\bar{c}=0.5$, $\forall i,k,n$.

\noindent \textbf{Examples of bases for directed graphs.}
For the sake of understanding the structure of the GFT basis vectors obtained with our methods, we start  considering the  simple directed graphs depicted  in  Fig. \ref{fig:allgraphs}, i.e. a directed graph composed of $N=15$ nodes with  three clusters, connected by a) $2$ directed links, b) $3$ directed links, and  c) a directed  cycle.
 As a first example, in Fig. \ref{fig:fdirtree} we report the basis vectors $\{\mx_k\}_{k=1}^{15}$ obtained through Algorithm \ref{algorithm:Alg_PAMAL_1}
 for graph (a) in Fig.  \ref{fig:allgraphs}.
The intensity of the vector entries is encoded in the color associated to each vertex.
Directed and undirected edges are represented by arrowed and continuous lines, respectively.
The order chosen to plot the basis vectors corresponds to increasing values of the directed variation $\text{GDV}(\mx_k)$ (reported on top of each subgraph). It is possible to notice that the basis vectors tend to identify clusters and, furthermore, the value assumed by the basis vectors within each cluster is exactly constant. This is a useful property in view of applications to unsupervised or semi-supervised clustering, where the label (signal) associated to each cluster is exactly constant within the cluster. This property does not hold with current methods based on the eigenvectors of either Laplacian or adjacency matrices, whose behavior within each cluster is only smooth but not exactly constant. To grasp the reason for this difference, it is worth noticing that, in case of undirected graphs, the above property is a consequence of having minimized an $\ell_1$-norm (see, e.g., (\ref{TV-undirected})), rather than an $\ell_2$-norm, as in the case of the Laplacian eigenvectors.  It is interesting to remark from Fig. \ref{fig:fdirtree} how there are three basis vectors that yield a zero directed variation. In particular, besides the constant vector $\bx_1$, vectors $\bx_2$ and $\bx_3$, even if not constant, yield zero variation just by assigning values to the entries of the cluster $\{11 \div 15\}$ smaller than the values of clusters $\{1 \div 5\}$ and $\{6 \div 10\}$. Since there is no directed edge between clusters $\{1 \div 5\}$ and $\{6 \div 10\}$, there are two ways to enforce the previous property, still maintaining vector orthogonality.
As a further example, let us consider graph (b) in Fig. \ref{fig:allgraphs}, where we added a directed link from node $7$ to node $5$.  From Fig. \ref{fig:fdirtree2} we observe that, in this case, the number of basis vectors having zero directed variation reduce to two, since the presence of the new directed link
leads to  only one  possible way, besides the constant vector,  to have $\text{GDV}=0$ still preserving basis orthogonality.
\begin{figure}[t!]
\centering
       \begin{subfigure}[t]{1.5\textwidth}\hspace{-1.05cm}
            {\includegraphics[width=3.8in,height=2.9in]{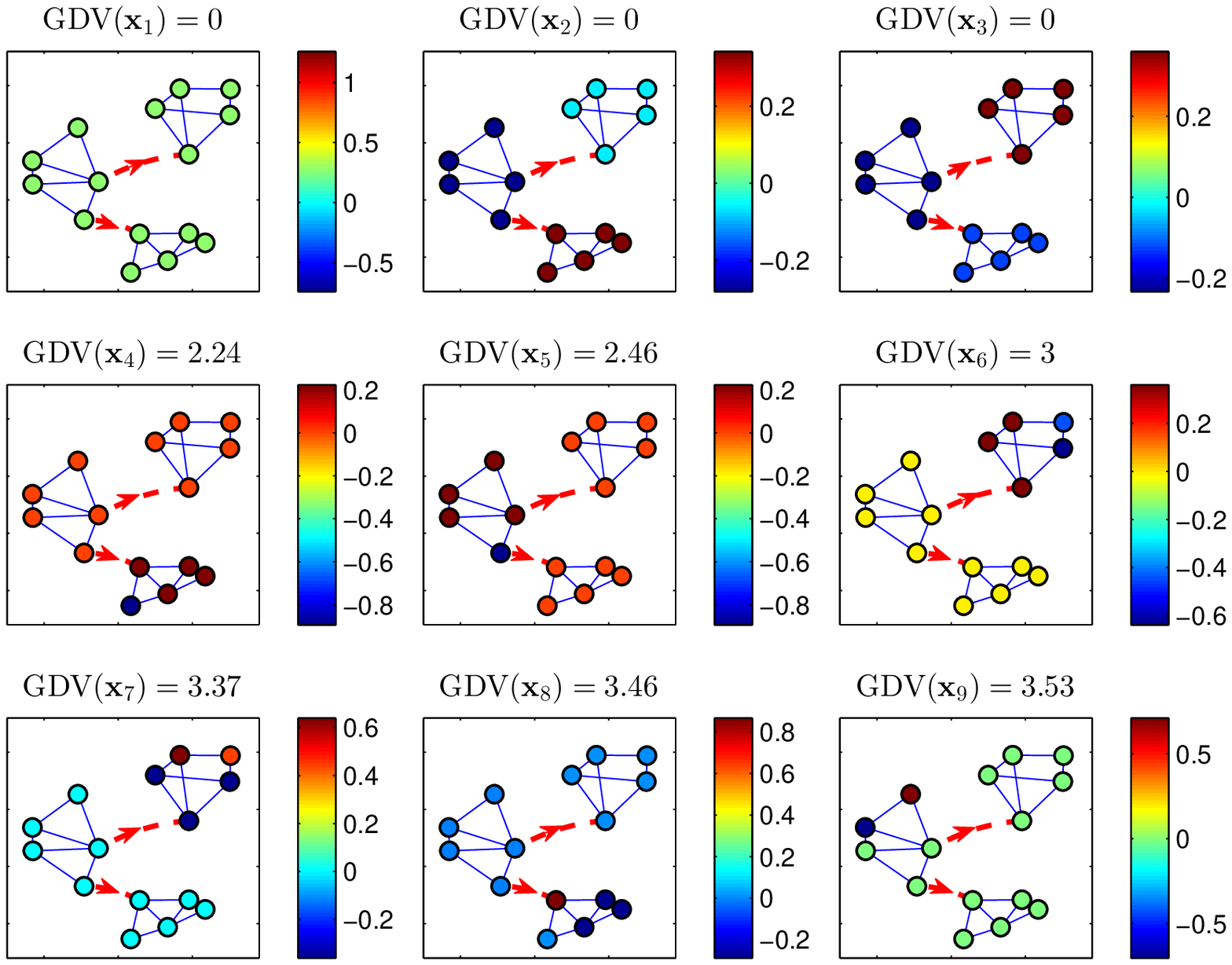}}
            \end{subfigure}\hspace{-0.3cm}
   \begin{subfigure}[b]{1.4\textwidth}\hspace{-0.65cm}
        \includegraphics[width=3.66in, height=1.95in]{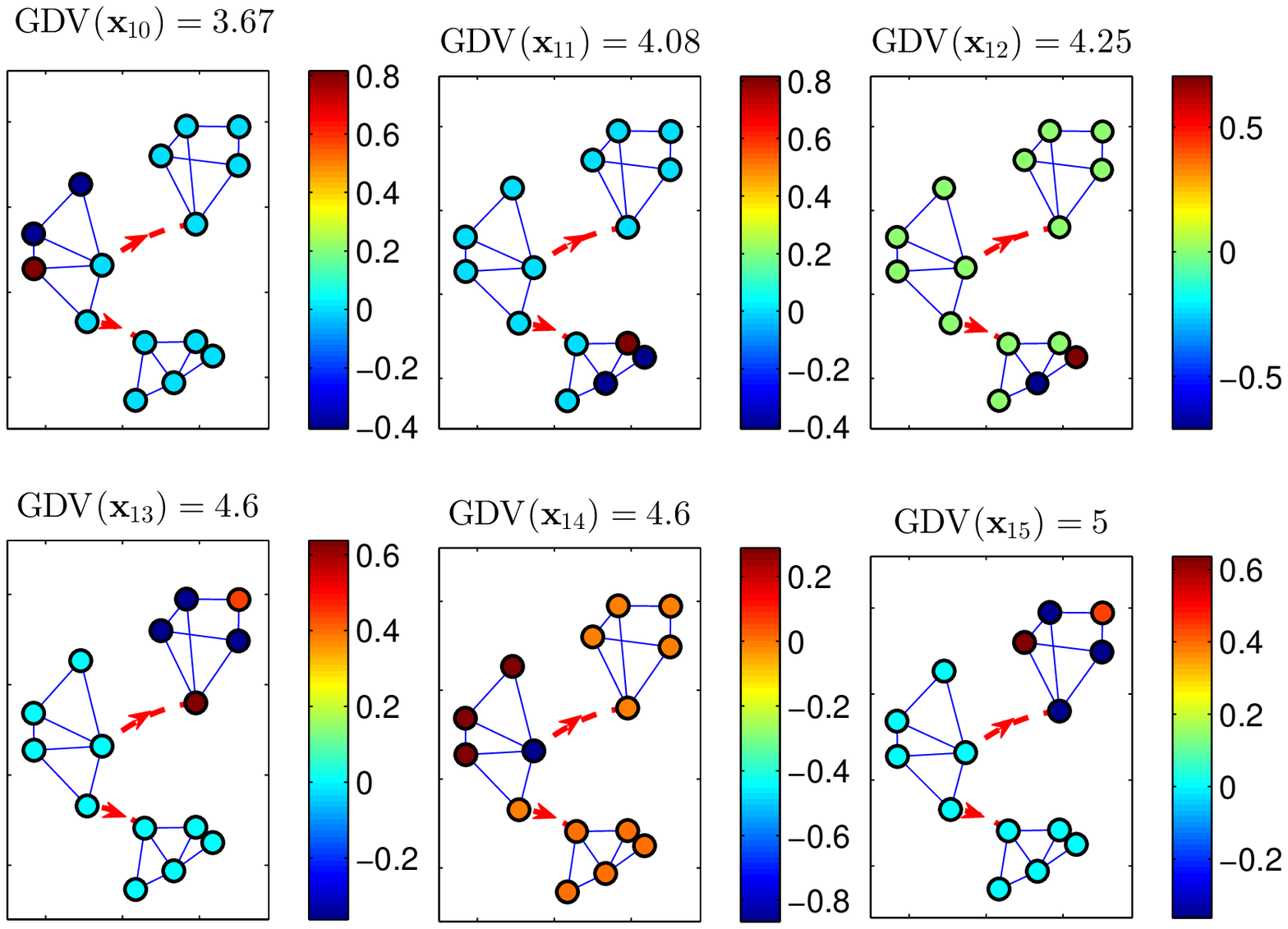}
           \end{subfigure}
    \caption{Optimal basis vectors $\mx_k$, $k=1,\ldots,15$  for Algorithm $2$ and the directed  graph in Fig. \ref{fig:graphs}.}\label{fig:fdirtree}
       \end{figure}
\begin{figure}[t!]
\centering
        \begin{subfigure}[t]{1.5\textwidth}\hspace{-0.6cm}
        \includegraphics[width=3.7in, height=2.9in]{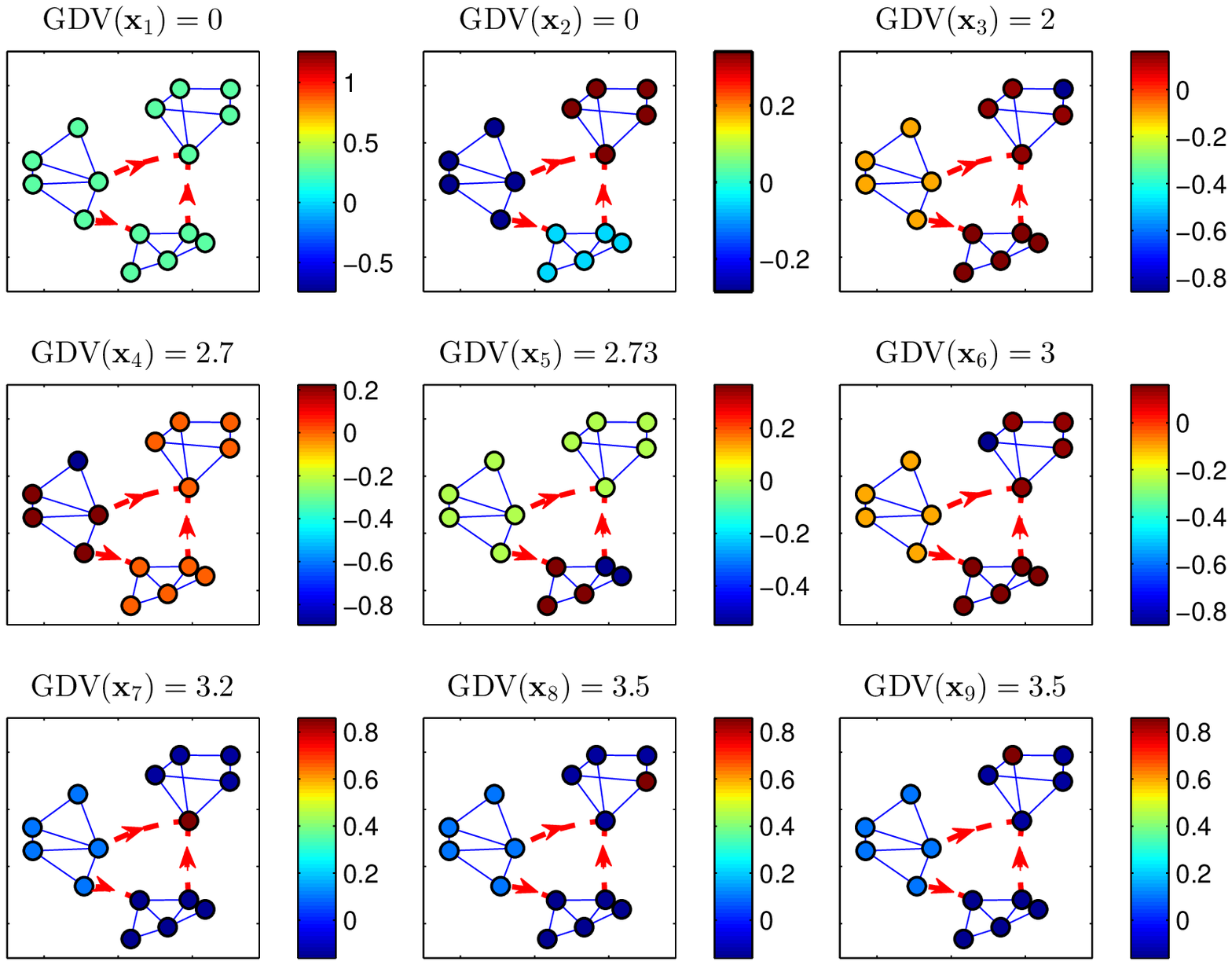}
            \end{subfigure}
            \hspace{-0.4cm}
              \begin{subfigure}[b]{1.5\textwidth}\hspace{-0.6cm}
        \includegraphics[width=3.73in, height=1.95in]{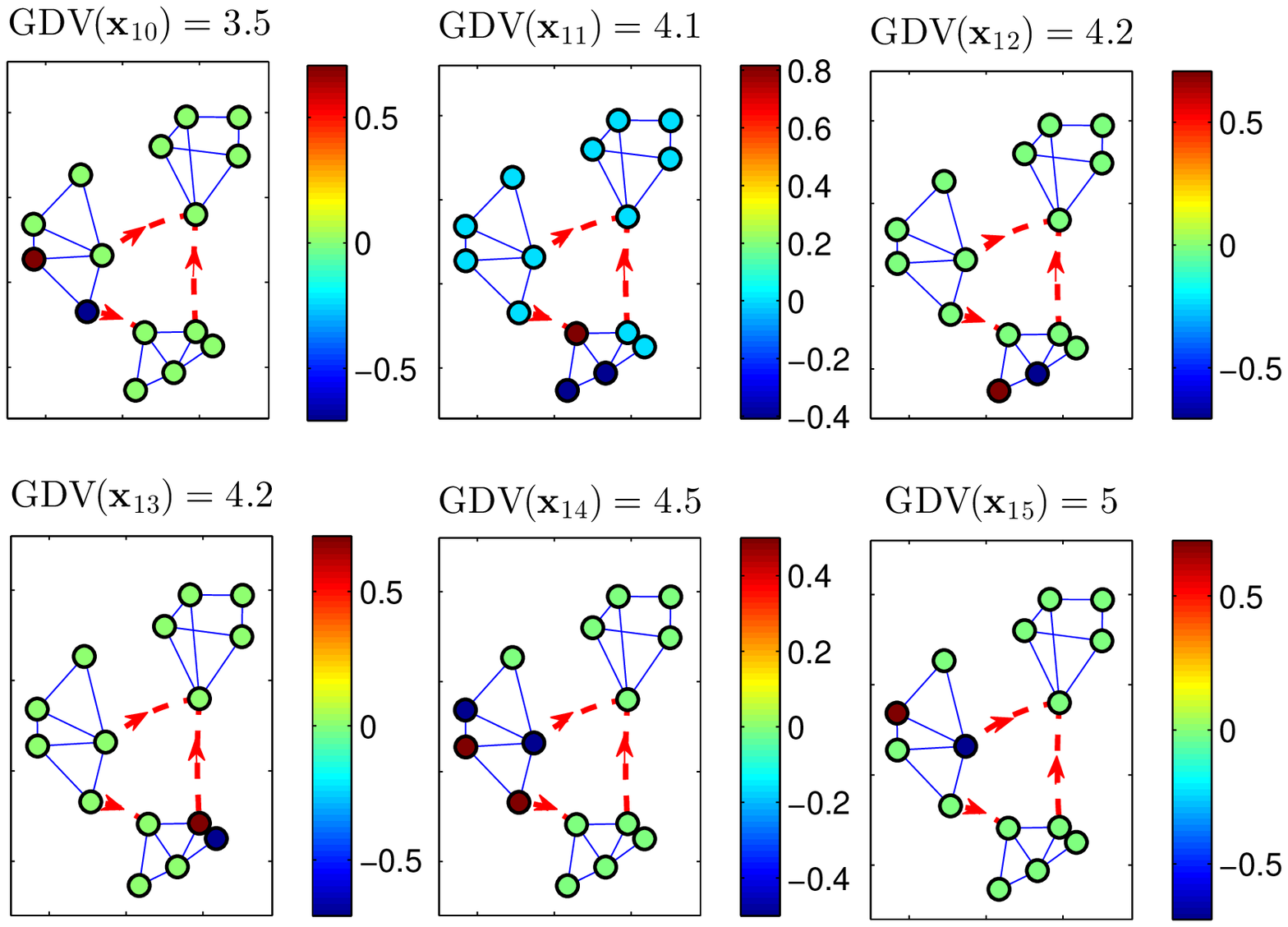}
           \end{subfigure}
    \caption{Optimal basis vectors $\mx_k$, $k=1,\ldots,15$  for Algorithm $2$ and the graph in Fig. \ref{fig:graphs_3}.}\label{fig:fdirtree2}
       \end{figure}
In Fig. \ref{fig:fcyclegraph}, we report the optimal basis, computed using Algorithm  \ref{algorithm:Alg_PAMAL_1},
for   the graph with a directed cycle depicted in Fig.  \ref{fig:graphcycle}. Interestingly, in this case, there can only be one vector that yields zero directed variation: the constant vector. In fact, the cyclical structure of the graph now prevents the existence of non-constant vectors able to null the directed variation.
The properties described above are a unique and an interesting consequence of the edge directivity. In fact, as can be observed from Fig. \ref{fig:f_und_graph}, the optimal bases for the corresponding undirected graph (obtained by simply removing edge directivity)  have only one vector with zero variation, the constant vector. Conversely, in the case shown before, we have had three, two, and one vectors yielding zero variation.

\noindent \textbf{Convergence test.}
Since the optimization problem $\mathcal{P}$ is non-convex, there is of course the possibility that the proposed methods fall into a local minimum. Furthermore, while PAMAL method guarantees convergence, SOC algorithm might also fail to converge because, theoretically speaking, there is no convergence analysis. To test what happens, we considered several independent initializations of both SOC and PAMAL algorithms in the search for a basis for the graph of Fig. \ref{fig:graphs}.  In Fig. \ref{fig:figconv}, we report the average behavior ($\pm$  the standard deviation) of the directed variation versus the iteration index $m$, which counts the overall number of  (outer and inner) iterations for Algorithm \ref{algorithm:Alg_SOC} and  \ref{algorithm:Alg_PAMAL_1}. The curves refer to $200$ independent initializations of algorithms SOC and PAMAL, using the same initialization for both.
We can observe that in all cases the algorithms converge but indeed there is a spread in the final variation, meaning that both methods can incur into local minima. Nonetheless, the spread is quite limited, which suggests that bases associated to different local minima behave similarly in terms of total variation.
Additionally,  since the PAMAL algorithm solves the orthogonality constrained, non-convex problem by  iteratively updating the primal variables and the multipliers, the objective function evaluated at each (inner and outer) iteration
does not necessarily follow a monotonous decay, as can be noticed by the lower subplot in Fig. \ref{fig:figconv}.\\
\begin{figure}[t!]%
\centering
      \begin{subfigure}[t]{1.5\textwidth}\hspace{-1.1cm}
            \includegraphics[width=3.8in, height=2.9in]{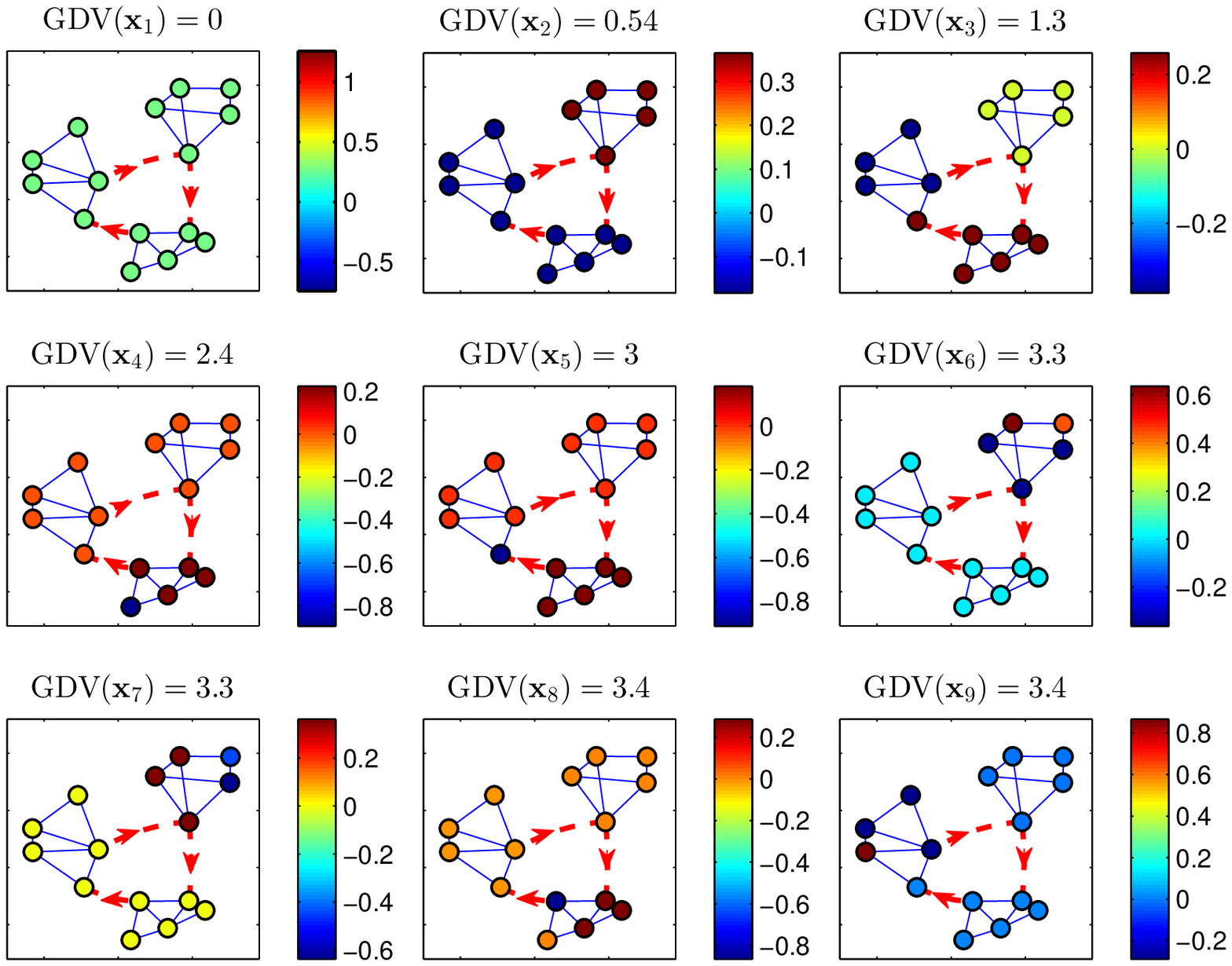}\vspace{-0.5cm}
         \end{subfigure}%
        \hspace{-0.7cm}%
        \begin{subfigure}[b]{1.4\textwidth} \hspace{-1.1cm}
           \includegraphics[width=3.83in, height=1.95in]{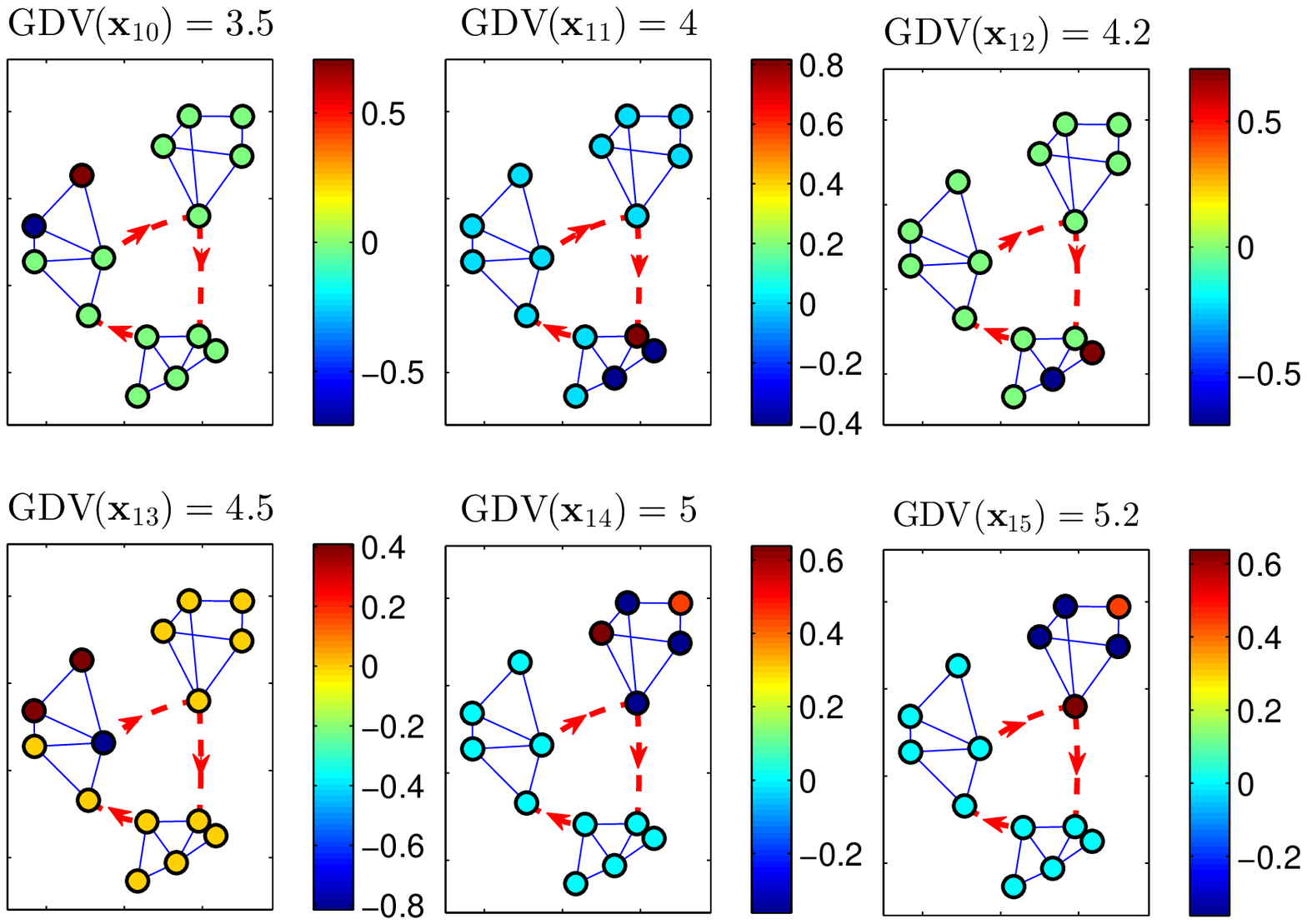}
           \end{subfigure}
    \caption{Optimal basis vectors $\mx_k$, $k=1,\ldots,15$  for Algorithm $2$ and the  graph in Fig. \ref{fig:graphcycle}.}\label{fig:fcyclegraph}
    \end{figure}
\begin{figure}[t!] %
\centering
       \vspace{0.3cm}\hspace{0.6cm}
    \begin{subfigure}[t]{1.5\textwidth}\hspace{-0.4cm}
        \includegraphics[width=3.5in, height=2.8in]{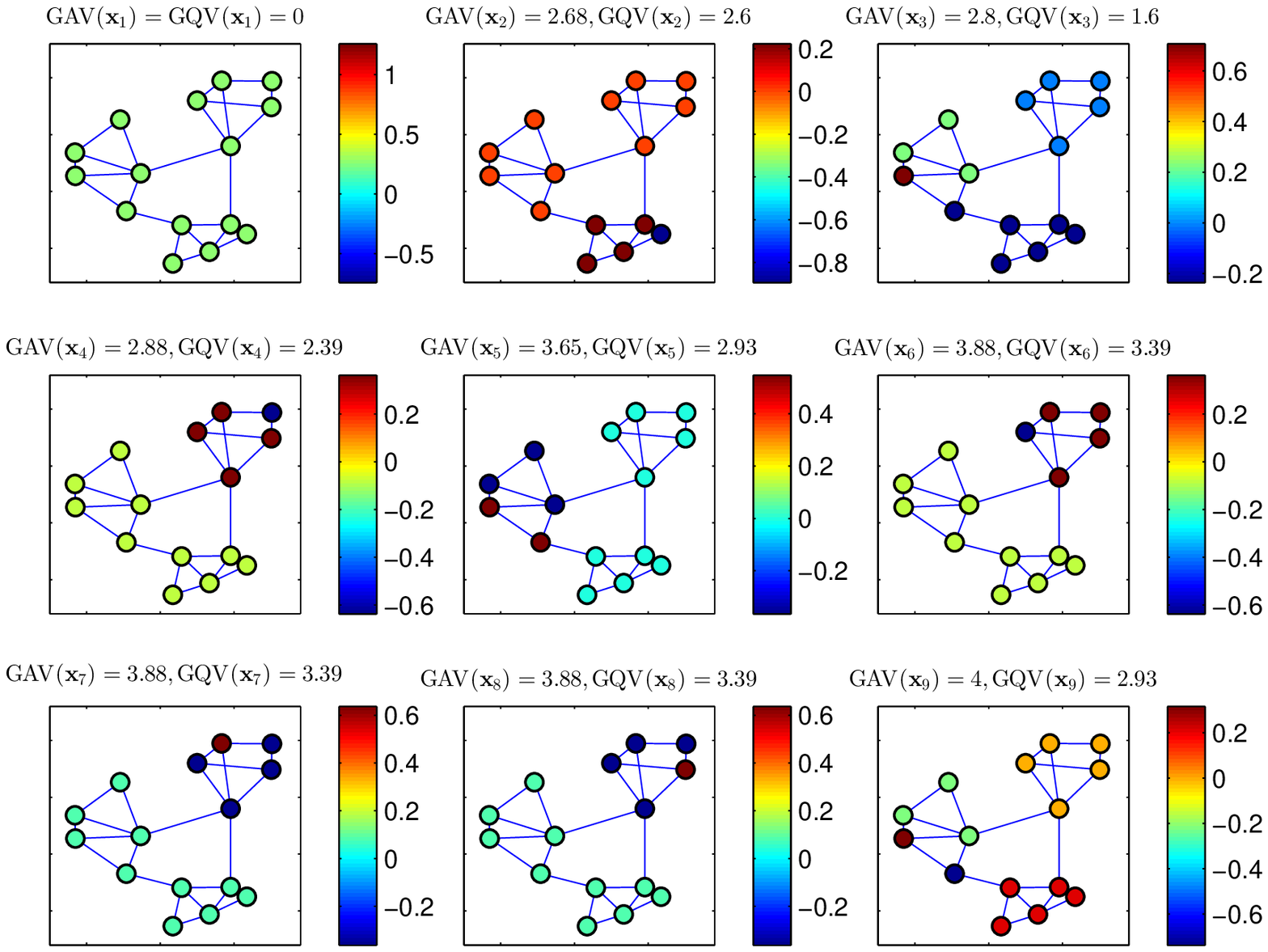}\vspace{0.37cm}
         \end{subfigure}%
         \hspace{-0.4cm}%
   \begin{subfigure}[b]{1.4\textwidth}\hspace{-0.65cm}
              \includegraphics[width=3.6in, height=1.75in]{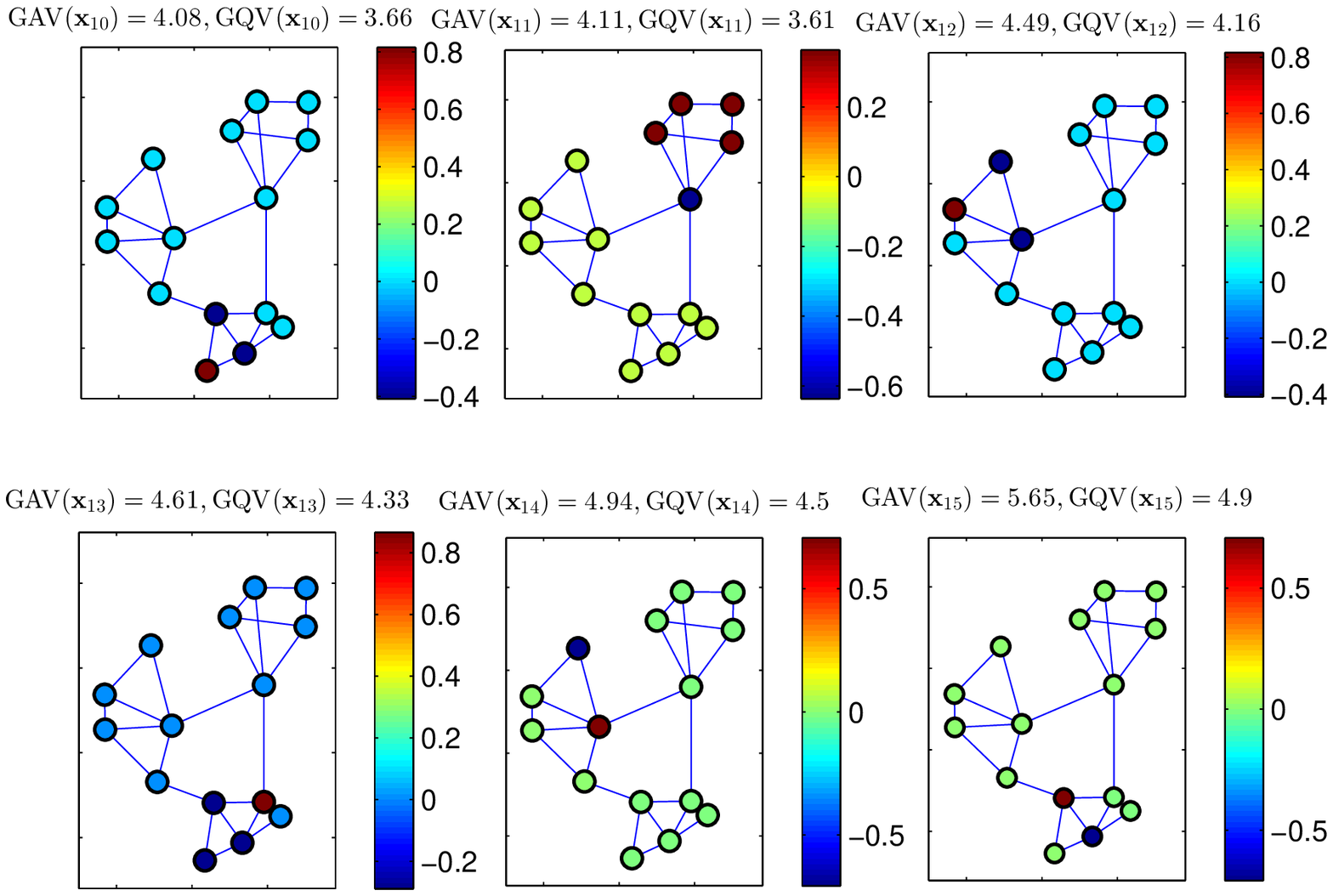}
           \end{subfigure}
    \caption{Optimal basis vectors $\mx_k$, $k=1,\ldots,15$  for Algorithm $2$ and the  undirected counterpart of the graph in Fig. \ref{fig:graphcycle}.}\label{fig:f_und_graph}
   \end{figure}
\noindent \textbf{Comparison with alternative GFT bases.}
We compare now the GFT basis found with our methods with the bases associated to either the Laplacian or the adjacency matrix, as proposed in \cite{Shuman2013},\cite{Moura2014} and references therein. To compare the results, we applied all algorithms to several independent realizations of random graphs. We chose as family of random graphs the so called scale-free graphs, as they are known to fit many situations of practical interest \cite{Albert_statisticalmechanics}. In the generation of random scale-free graphs, it is possible to set the minimum degree $d_{min}$ of each node.
To  compare our method with the GFT definition proposed in \cite{Moura2013},  since the eigenvectors of an asymmetric matrix can be complex and the directed total variation $\text{GDV}$, as defined in (\ref{TV-directed}), does not represent a valid metric for complex vectors, we restricted the comparison to {\it undirected} scale-free graphs, in which case the adjacency and Laplacian matrices are real and symmetric, so that  their eigenvectors are real.
   In the sequel, we will use the notations $\text{GAV}(\mX):= \sum_{k=1}^{N} \text{GAV}(\mx_k)$ and  $\text{GQV}(\mX):= \sum_{k=1}^{N} \text{GQV}(\mx_k)$ to denote, respectively,  the total graph absolute and quadratic  variation of a matrix $\mX$.
 In Fig. \ref{fig:figMouramed}, we compare the following metrics: a)  $\text{GAV}(\mX^{*})$, derived by solving problem $\mathcal{P}$  through the SOC and PAMAL methods; b) $\text{GAV}(\mV)$,  where $\mV$ are the eigenvectors of the adjacency matrix according to the GFT defined in (\ref{V-1x}); c) $\text{GAV}(\mU)$,  where $\mU$ are the eigenvectors of the Laplacian  matrix by assuming the GFT as in  (\ref{Ux}), that for  undirected graphs is equivalent to the  GFT defined in (\ref{V-L}). More specifically, Fig. \ref{fig:figMouramed} shows the previous metrics vs. the minimum degree of the graph averaged over $100$ independent realizations of scale-free graphs of $N=20$ nodes.
As we can notice from  Fig. \ref{fig:figMouramed},  the bases built using SOC and PAMAL algorithms yield a significantly lower total variation than the conventional bases built with either adjacency or Laplacian eigenvectors. This is primarily due to the fact that our optimization methods tend to assign constant values within each cluster.  Finally, in Fig. \ref{fig:fig6new}  we compare the alternative basis vectors using as performance metric the  $\text{GQV}$. So, in Fig. \ref{fig:fig6new} we report the $\text{GQV}(\mX^{*})$ metric derived from the SOC and PAMAL methods with $\text{GQV}(\mV)$ and
$\text{GQV}(\mU)$  obtained, respectively, from the eigenvectors of the adjacency and the Laplacian matrix. Again, the results are averaged over
$100$ independent realizations of scale-free graphs, vs. the average minimum degree, under the same settings  of Fig. \ref{fig:figMouramed}.
 Interestingly,  even if our basis vectors  $\mX^{*}$ do not coincide with $\mV$ or $\mU$, they provide the same $\text{GQV}$, within negligible numerical inaccuracies. Indeed, the invariance of the metric $\text{GQV}(\mX)$,  for any square, orthogonal matrix $\mX$, can be easily proved  from the equality  $\text{GQV}(\mX)= \sum_{k=1}^{N} \mx^T_k \mL \mx_k=\text{trace}(\mX^T \mL \mX)$,  by observing that  $\text{trace}(\mX^T \mL \mX)=\text{trace}(\mL)$ for any orthogonal matrix $\mX$.
 Interestingly, this implies that, for undirected graphs, our orthogonal matrix $\mX^{*}$ can be obtained by applying an orthogonal transform to the Laplacian eigenvectors basis.
 
\begin{figure}[t!]
\centering
              \includegraphics[width=3.3in]{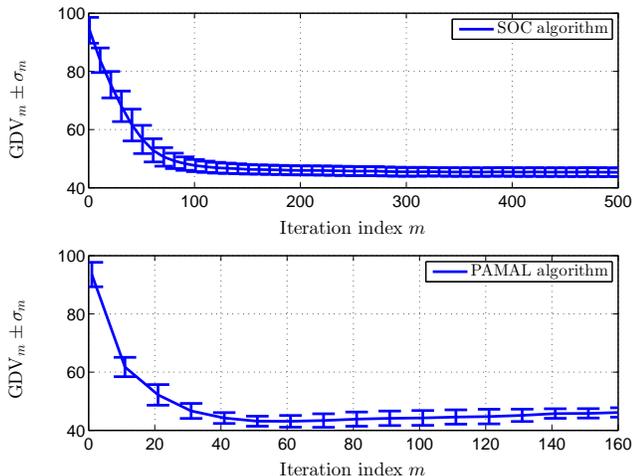}
           \caption{Average directed variation ($\pm$  the standard deviation) for SOC and PAMAL methods vs. the iteration index $m$ for the graph of Fig. \ref{fig:graphs}, by averaging  over $200$ random initializations of the algorithms.}\label{fig:figconv}
\end{figure}
 \begin{figure}[!] \centering
\includegraphics[width=3.3in,height=2.6in]{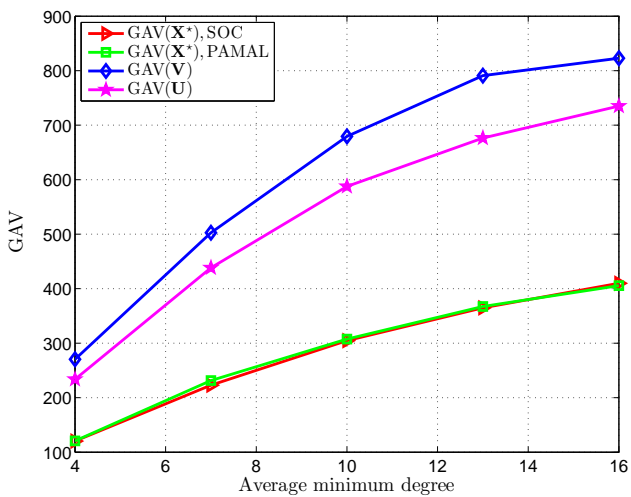}
\caption{Average absolute total variation versus the average minimum degree according to alternative GFT definitions for undirected scale-free graphs with $N=20$ nodes.}\label{fig:figMouramed}
\end{figure}

\noindent \textbf{Complexity issues.}
Clearly, looking at both SOC and PAMAL methods, complexity is a non trivial issue which deserves further investigations, especially when the size of the graph increases. To get an idea of computing time,  in Fig. \ref{fig:figexec}
we report the execution time of both SOC and PAMAL algorithms, as a function of the number of vertices in the graph.
The results have been obtained
running a non-compiled  Matlab program, with no optimization of the parameters involved, by setting $\rho^{1}=\beta=20$. The program ran on a laptop
having a processor Intel Core i7-4500, CPU 1.8, 2.4 GHz. The graphs under test were generated as geometric random graphs with equal percentage of directed links as $N$ increases.

 \begin{figure}[!] \centering
\includegraphics[width=3.3in,height=2.6in]{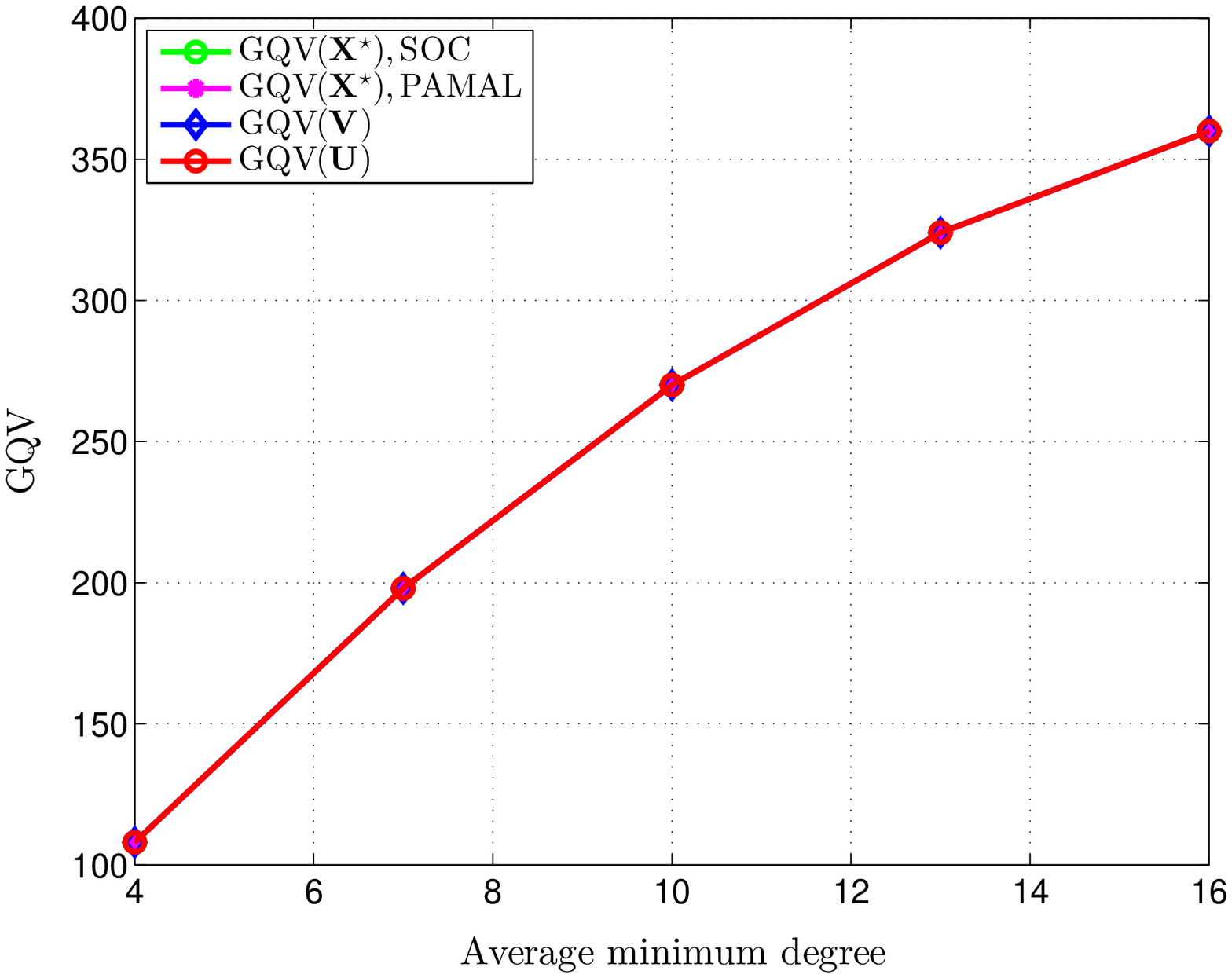}
\caption{Average $\text{GQV}$ versus the average minimum degree according to alternative GFT definitions for undirected scale-free graphs with $N=20$ nodes.}\label{fig:fig6new}
\end{figure}

\noindent \textbf{Examples with real networks.}
As an application to real graphs, in Fig. \ref{fig:figmazzini}  we considered the directed  graph
obtained from the street map of Rome, incorporating the true directions of traffic lanes in the area around Mazzini square. The graph is composed of $239$ nodes. Even though, the scope of this paper is to propose a method to build a GFT basis, so that we do not dig further into applications, this an example that has interesting applications of GSP. The problem in this case is to build a map of vehicular traffic in a city, starting from a subset of measurements collected along road side units or sent by cars equipped with ad hoc equipment. The problem can be interpreted as the reconstruction of the entire graph signal from a subset of samples and then it builds on graph sampling theory \cite{Tsit_Barb_PDL}.
In Fig. \ref{fig:fdirsquare}  we report some basis vectors obtained by using Algorithm $2$ with $\rho^1=10$. We can observe that the basis vectors highlight clusters, while capturing the edges' directivity.

\noindent \textbf{Balanced total  variation.} In some cases, the solution of the total variation problem in (\ref{prob_main}) can cut the graph in subsets of very different  cardinality. As an extreme case, it may be not uncommon to have a subset composed of only one node and the other set containing all the rest of the network. To prevent such a behavior,  Algorithm \ref{algorithm:Alg_balanced}
aims at minimizing the   balanced total variation. An example of its application   to the graph of Fig. \ref{fig:figmazzini} is reported in Fig. \ref{fig:fdirsquarebal1}, where we show  some basis vectors computed using Algorithm  \ref{algorithm:Alg_balanced}. Comparing these vectors with the corresponding ones obtained with PAMAL algorithm, see, e.g. Fig. \ref{fig:fdirsquare}, we can see how clusters of single nodes are now avoided.\vspace{-0.2cm}

\begin{figure}[!]
\centering
               \includegraphics[width=3.3in,height=2.6in]{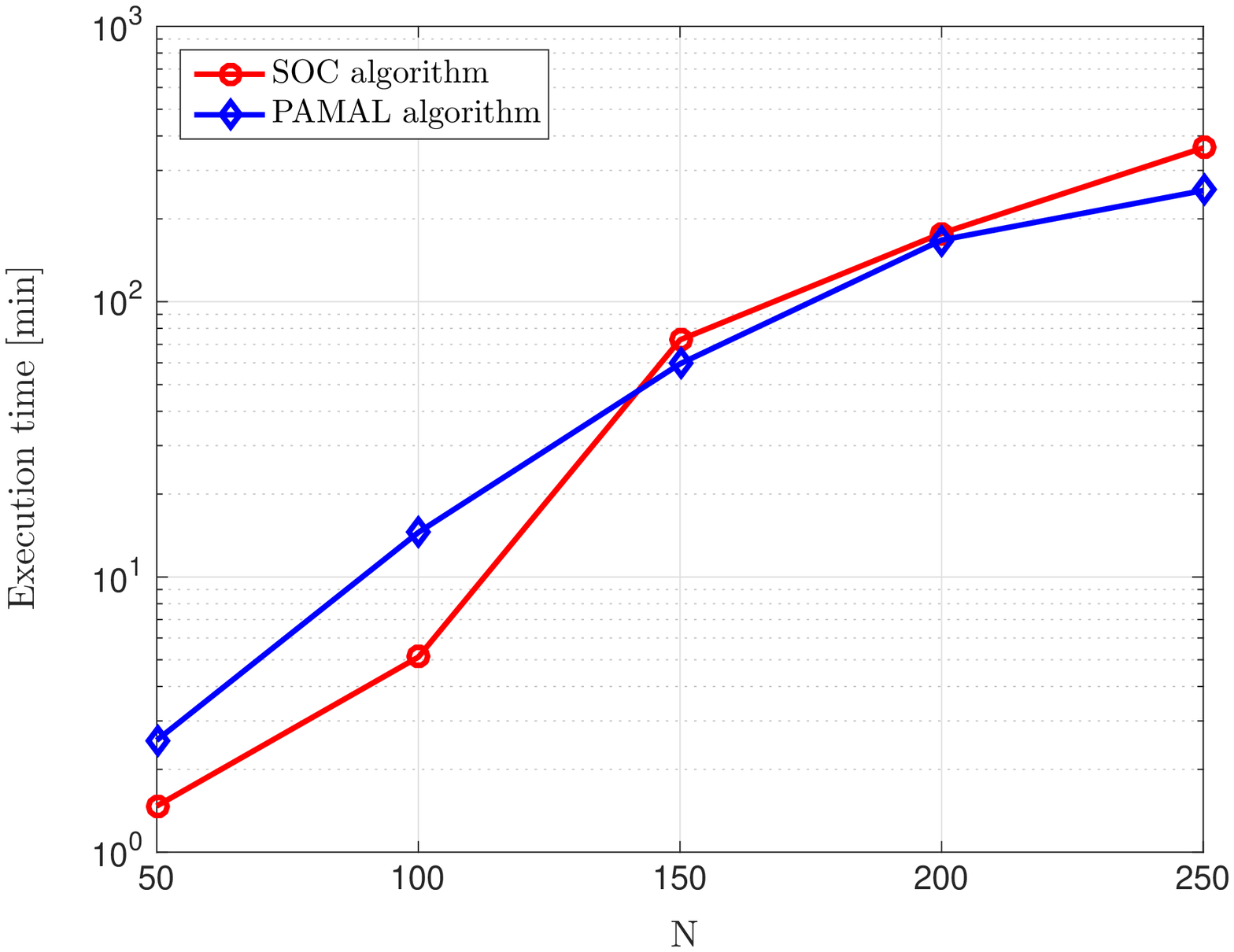}
          \caption{Execution time vs. the number of nodes for RGGs with $25\%$ of directed links and $\beta=\rho^1=20$.}\label{fig:figexec}
\end{figure}

\begin{figure}[!]
    \centering
           \includegraphics[width=3.5in,height=3.42in,clip=true]{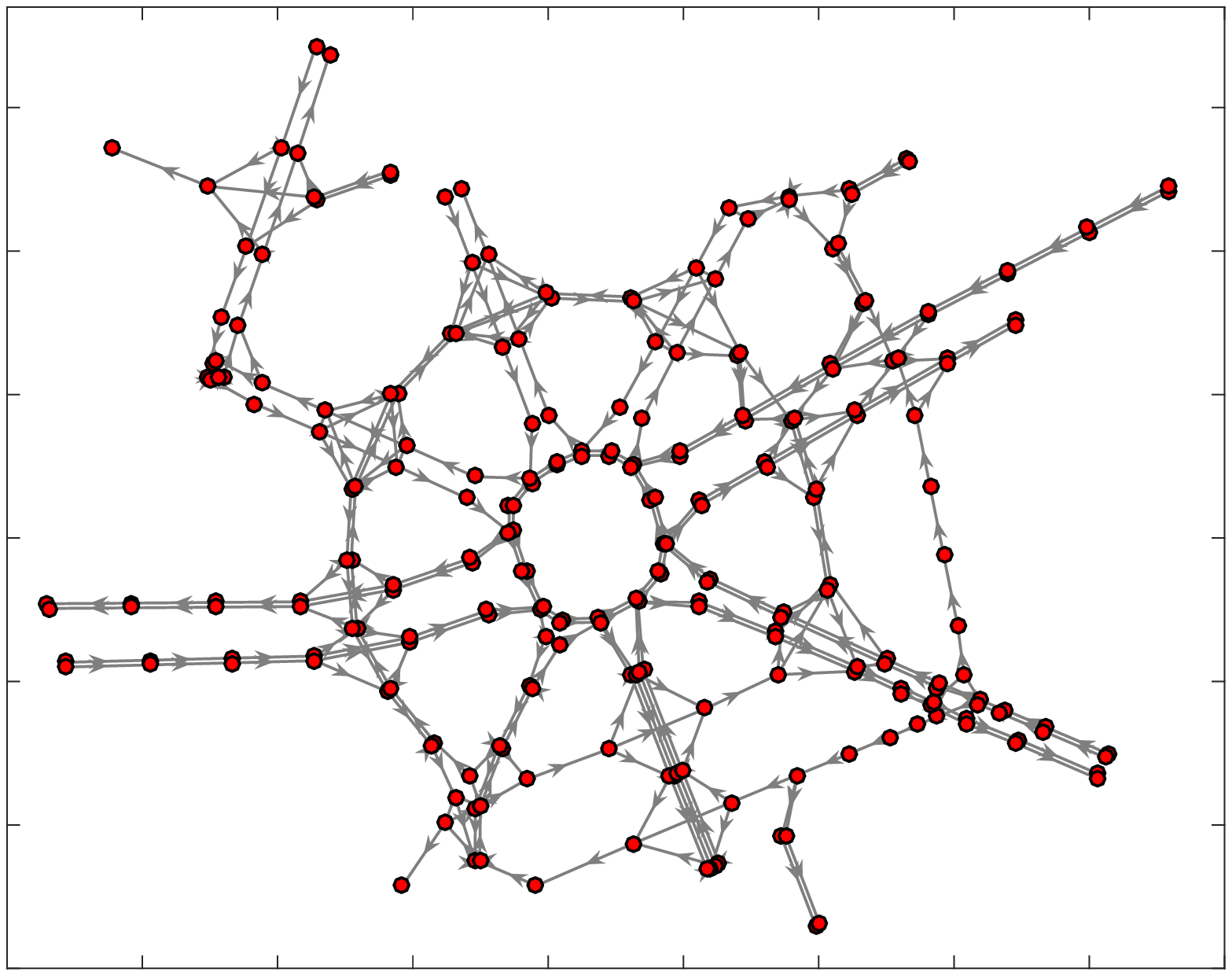}\vspace{0.6cm}
        \caption{Directed graph associated to street map of Rome (Piazza Mazzini).}\label{fig:figmazzini}
\end{figure}

\section{Conclusion}
\label{section:conclusions}
In this paper we have proposed an alternative approach to build an orthonormal basis for the Graph Fourier Transform (GFT). The approach considers the general case of a directed graph and then it includes the undirected case as a particular example. The search method starts from the identification of an objective function and then looks for an orthonormal basis that minimizes that function. More specifically, motivated by the need to detect clustering behaviors in graph signals, we chose as objective function the cut size. We showed that this approach leads, without loss of optimality, to the minimization of a function that represents a directed total variation of graph signals, as it captures the edges' directivity. Interestingly, in case of undirected graphs, this function converts into an $\ell_1$-norm total variation, which represents the graph (discrete) counterpart of the $\ell_1$-norm total variation that plays a key role in the classical Fourier Transform of continuous-time signals \cite{Mallat}. We compared our basis vectors with the eigenvectors of either the Laplacian or adjacency matrix, assuming as performance metric either our graph absolute variation or the graph quadratic variation. As expected, our method outperforms the other methods when using the absolute variation, as it is built by minimizing that metric. However, what has been interesting to see was that our basis performs as well as the alternative basis when we assumed as performance metric the graph quadratic variation.
Before concluding, we wish to point out that, as always, our alternative approach to build a GFT basis has its own merits and shortcomings when compared to alternative approaches. For example, having restricted the search to the real domain, differently from available methods, our method fails to find  the complex exponentials as the GFT basis in the case of circular graphs. Furthermore, other methods like the ones in \cite{Moura2013} starting from the identification of the adjacency matrix as the shift operator, are more suitable than our approach to devise a filtering theory over graphs.\vspace{-0.2cm}

\begin{figure*}
\centering
\begin{subfigure}[b]{0.49\textwidth}\hspace{-1.5cm}
\centering
               \includegraphics[width=9.4cm,height=7.0cm,bb=50 226 542 600]{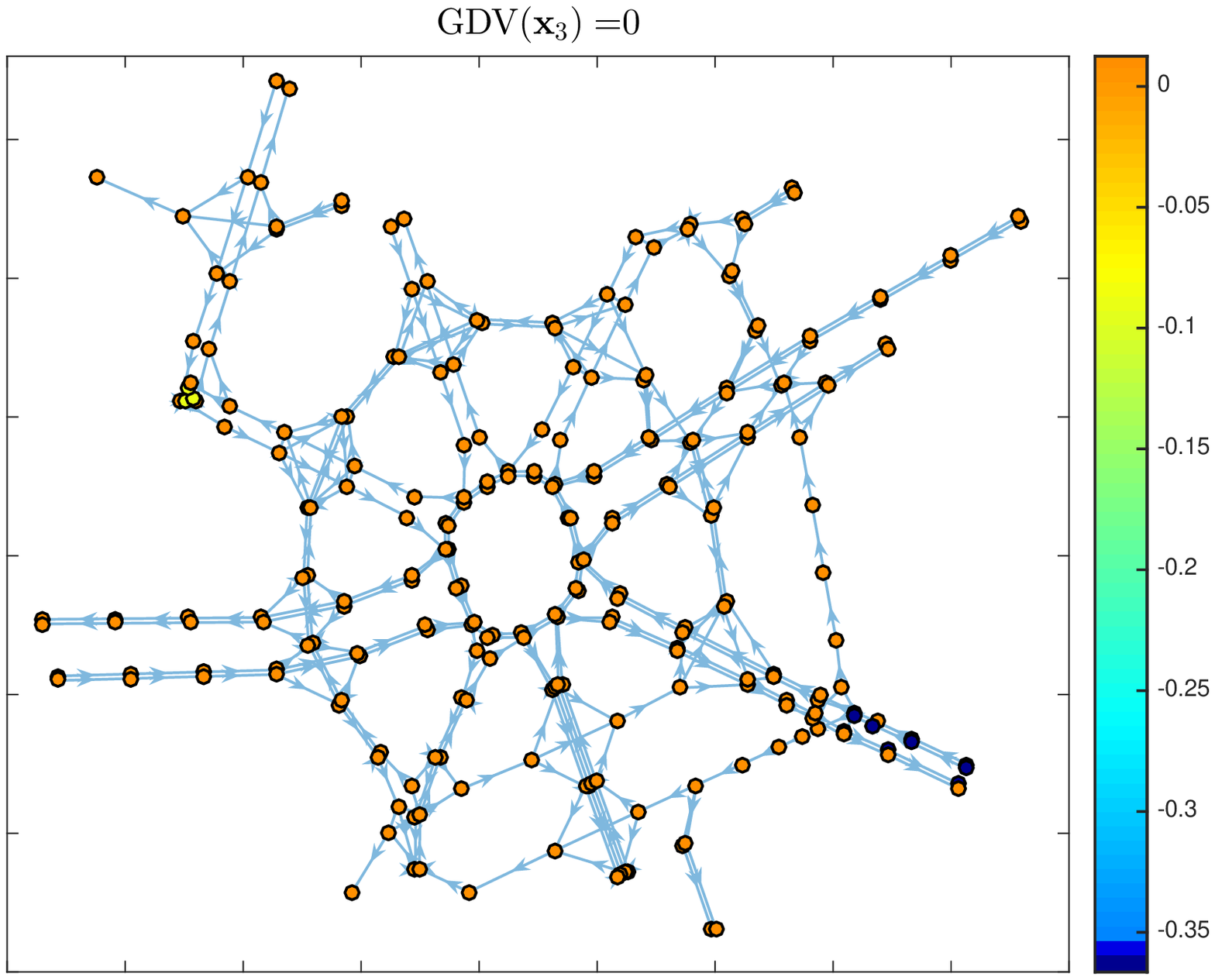}
     \end{subfigure}
\begin{subfigure}[b]{0.5\textwidth}
\centering
              \includegraphics[width=9.4cm,height=7.0cm,scale=1,bb=35 226 542 600]{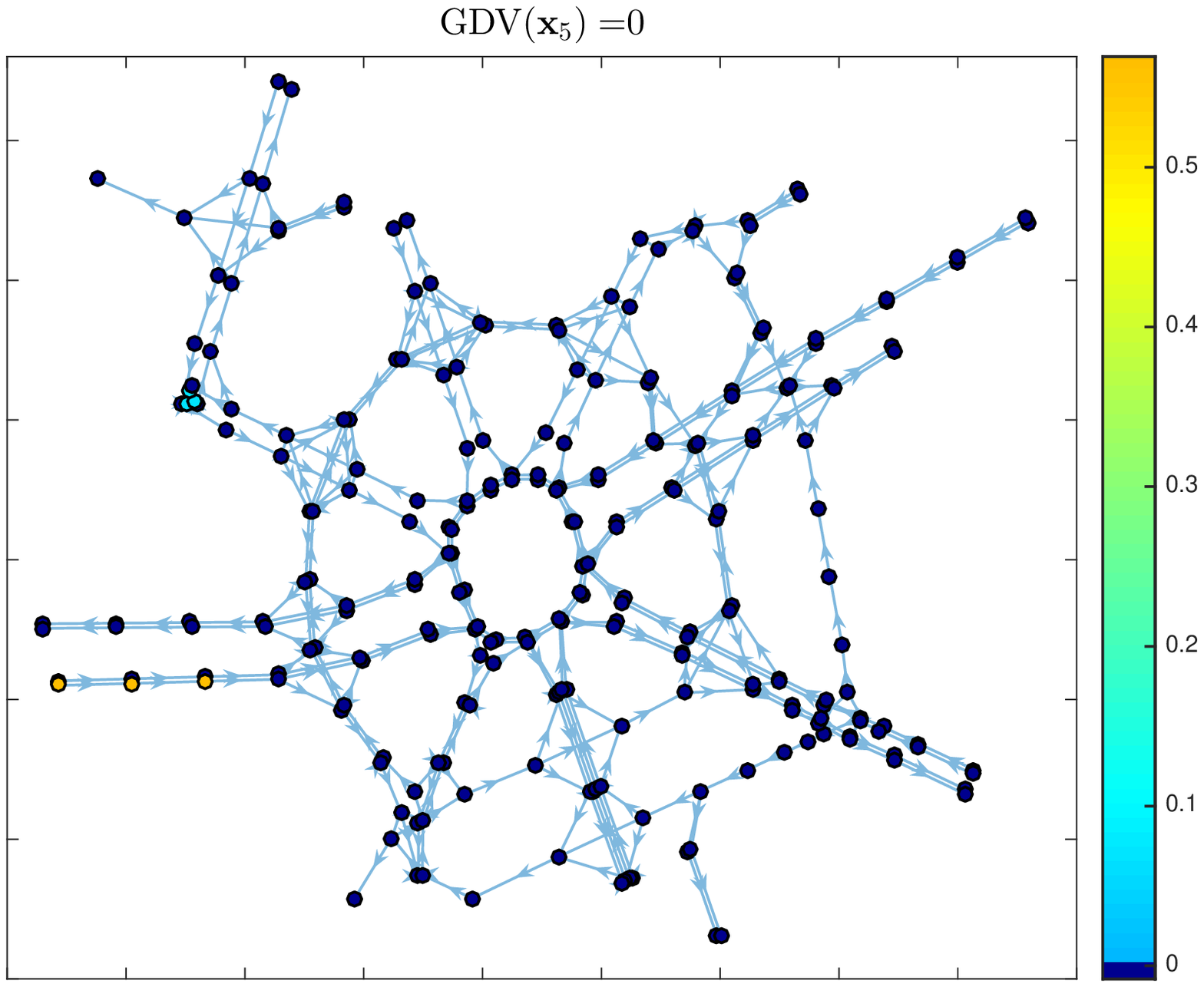}
\end{subfigure}
       \vskip \baselineskip
\begin{subfigure}[b]{0.49\textwidth}\hspace{-1.5cm}
\centering
               \includegraphics[width=9.4cm,height=7.0cm,bb=50 226 542 600]{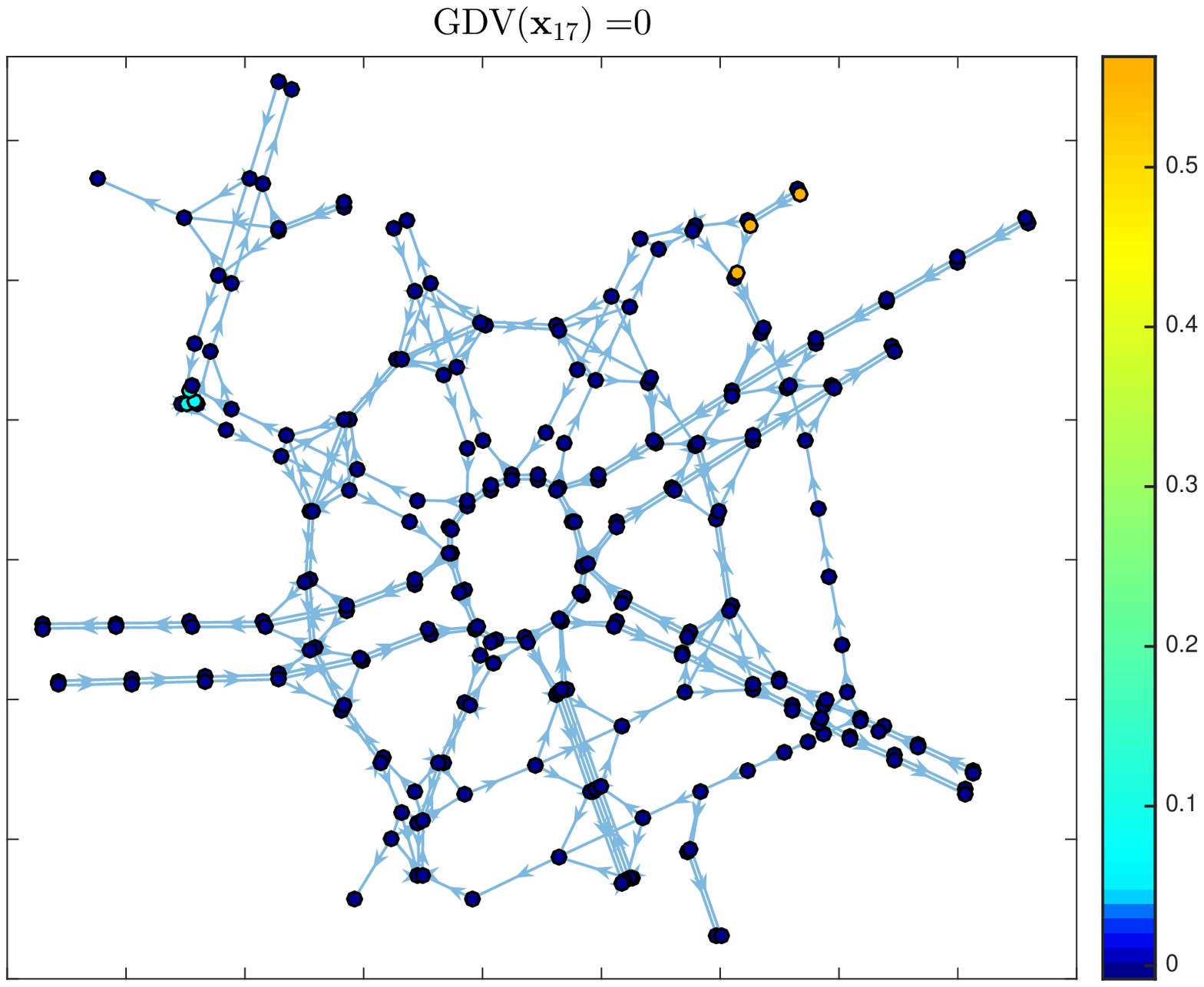}
      \end{subfigure}
\begin{subfigure}[b]{0.5\textwidth}
\centering
              \includegraphics[width=9.4cm,height=7.0cm,scale=1,bb=35 226 542 600]{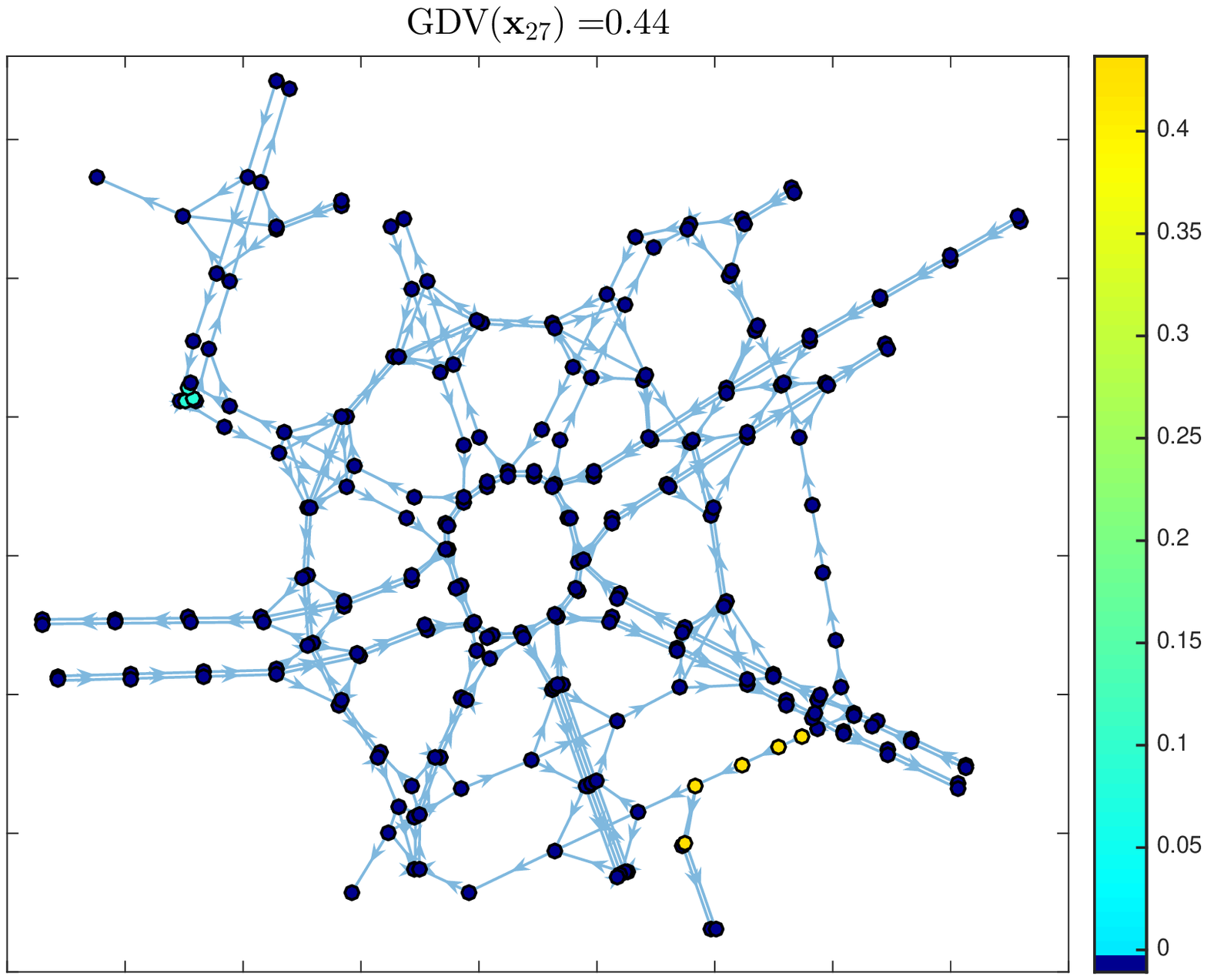}
\end{subfigure}
 \vskip \baselineskip
\begin{subfigure}[b]{0.49\textwidth}\hspace{-1.5cm}
\centering
               \includegraphics[width=9.4cm,height=7.0cm,bb=50 226 542 600]{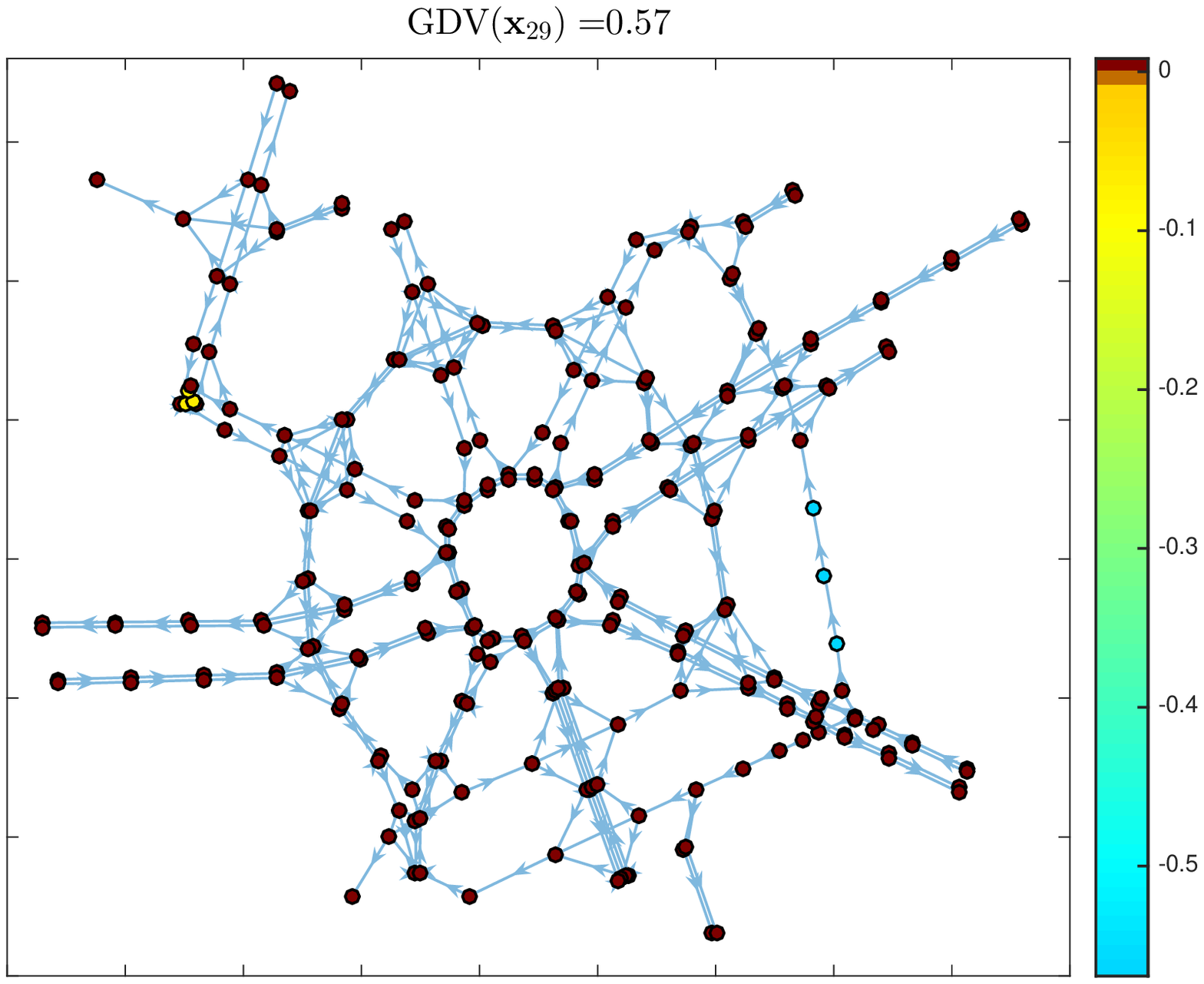}
      \end{subfigure}
\begin{subfigure}[b]{0.5\textwidth}
\centering
              \includegraphics[width=9.4cm,height=7.0cm,scale=1,bb=35 226 542 600]{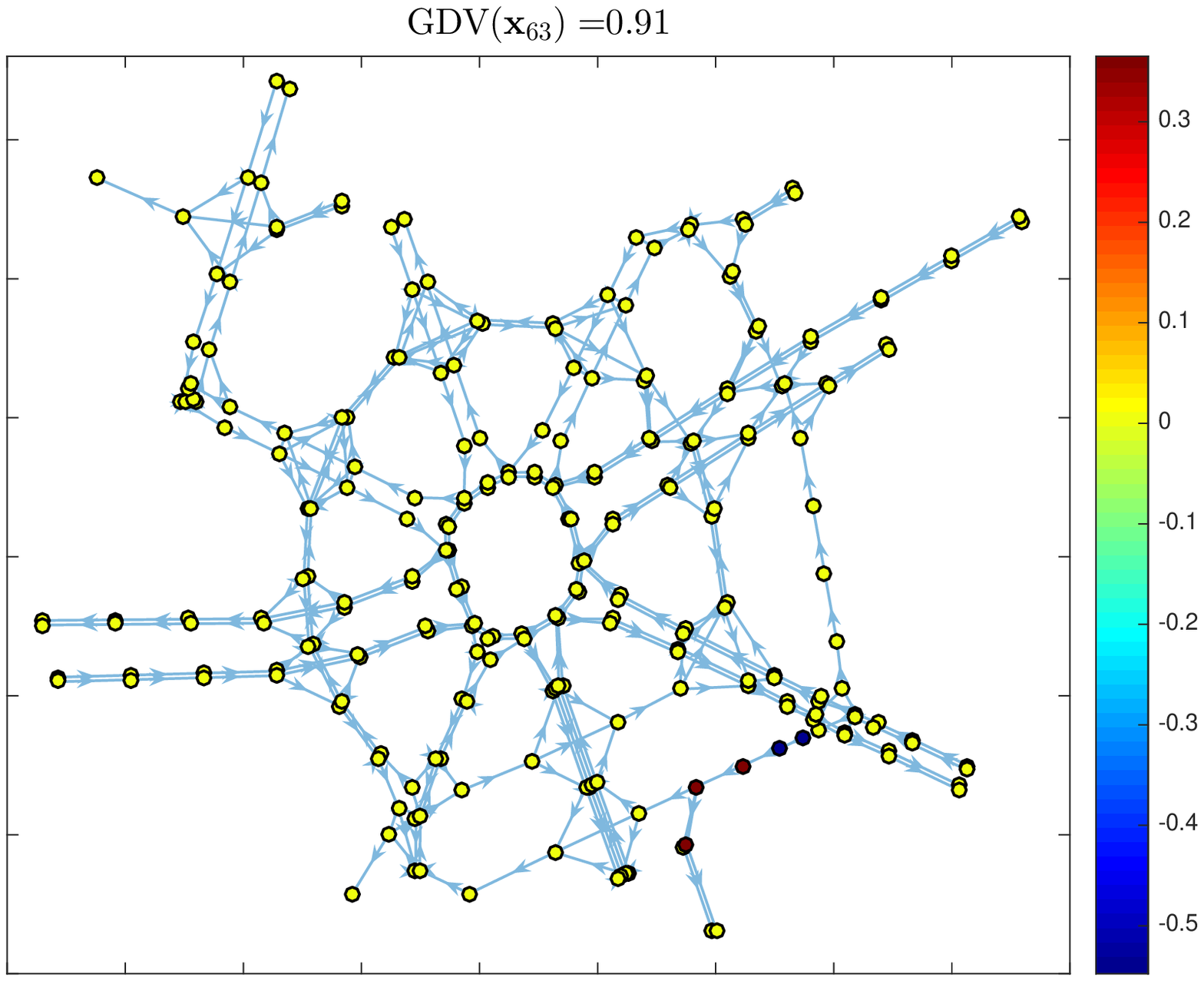}
\end{subfigure}\vspace{0.5cm}
\caption{Optimal basis vectors $\mx_k$, $k=3,5,17,27,29,63$ for Algorithm $2$ and the graph in Fig. \ref{fig:figmazzini}.}\label{fig:fdirsquare}
\end{figure*}

\begin{figure*}
\centering
\begin{subfigure}[b]{0.49\textwidth}\hspace{-1.5cm}
\centering
               \includegraphics[width=9.4cm,height=7.0cm,bb=50 226 542 600]{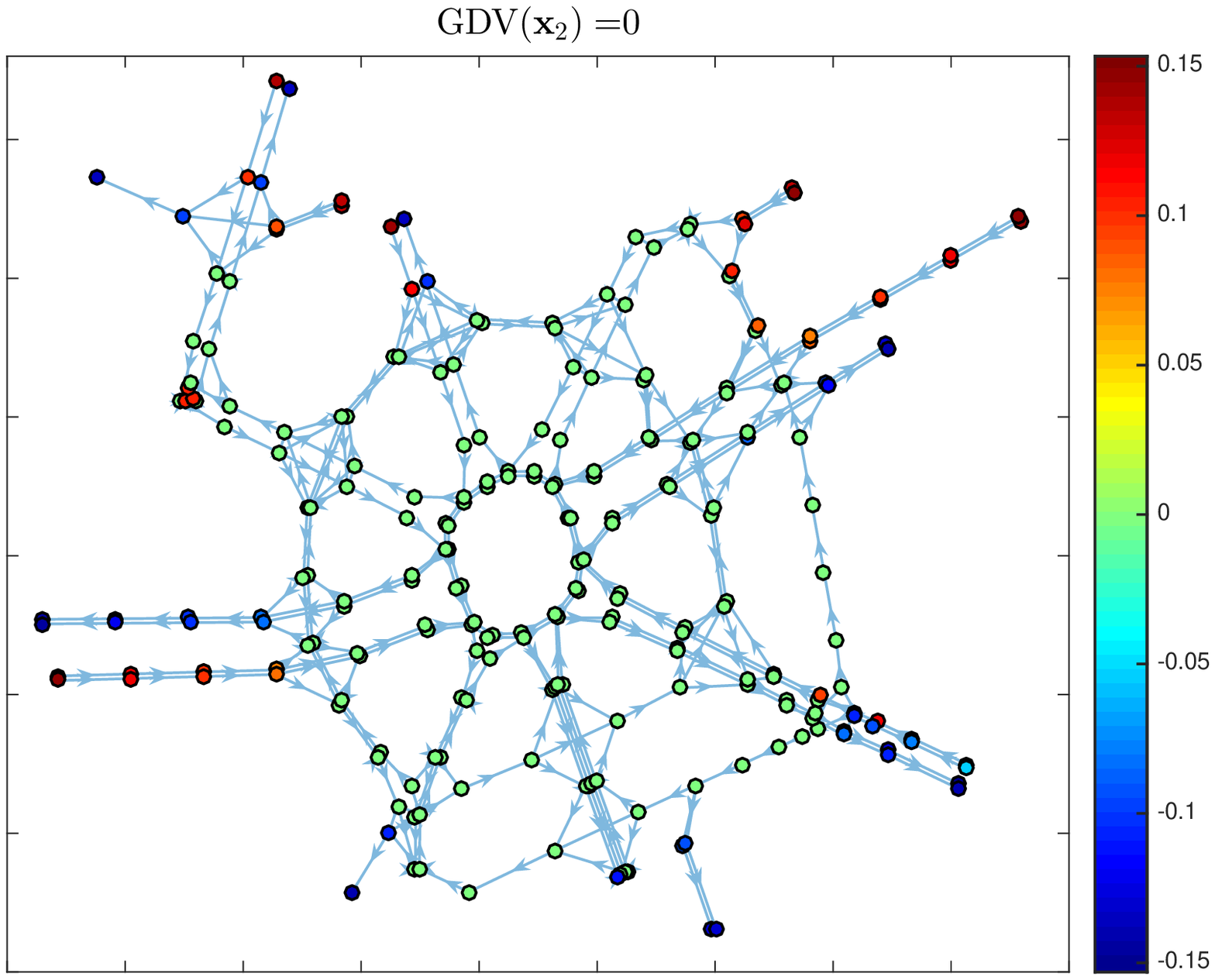}
     \end{subfigure}
\begin{subfigure}[b]{0.5\textwidth}
\centering
              \includegraphics[width=9.4cm,height=7.0cm,bb=35 226 542 600]{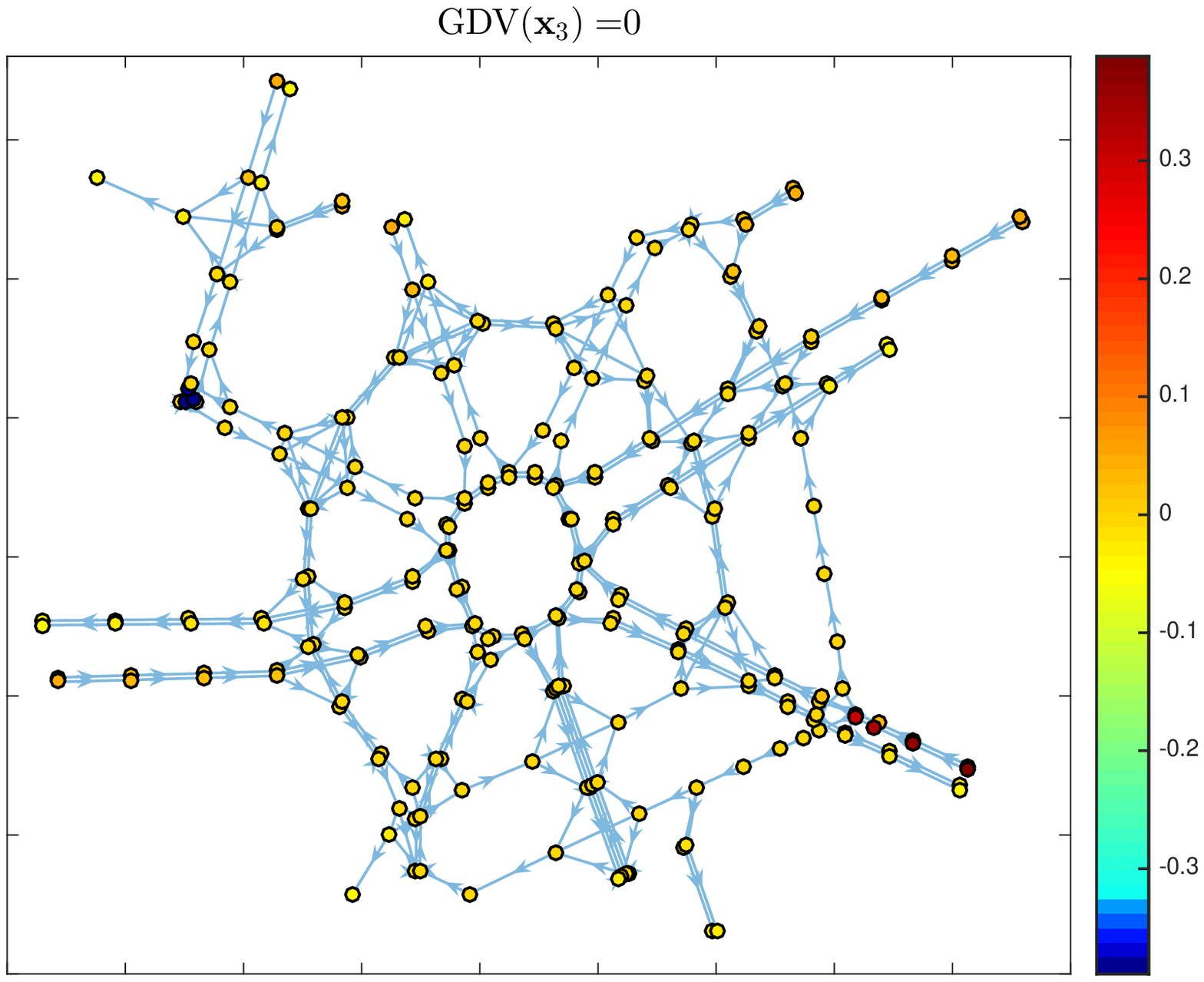}
\end{subfigure}
       \vskip \baselineskip
\begin{subfigure}[b]{0.49\textwidth}\hspace{-1.5cm}
\centering
               \includegraphics[width=9.4cm,height=7.0cm,bb=50 226 542 600]{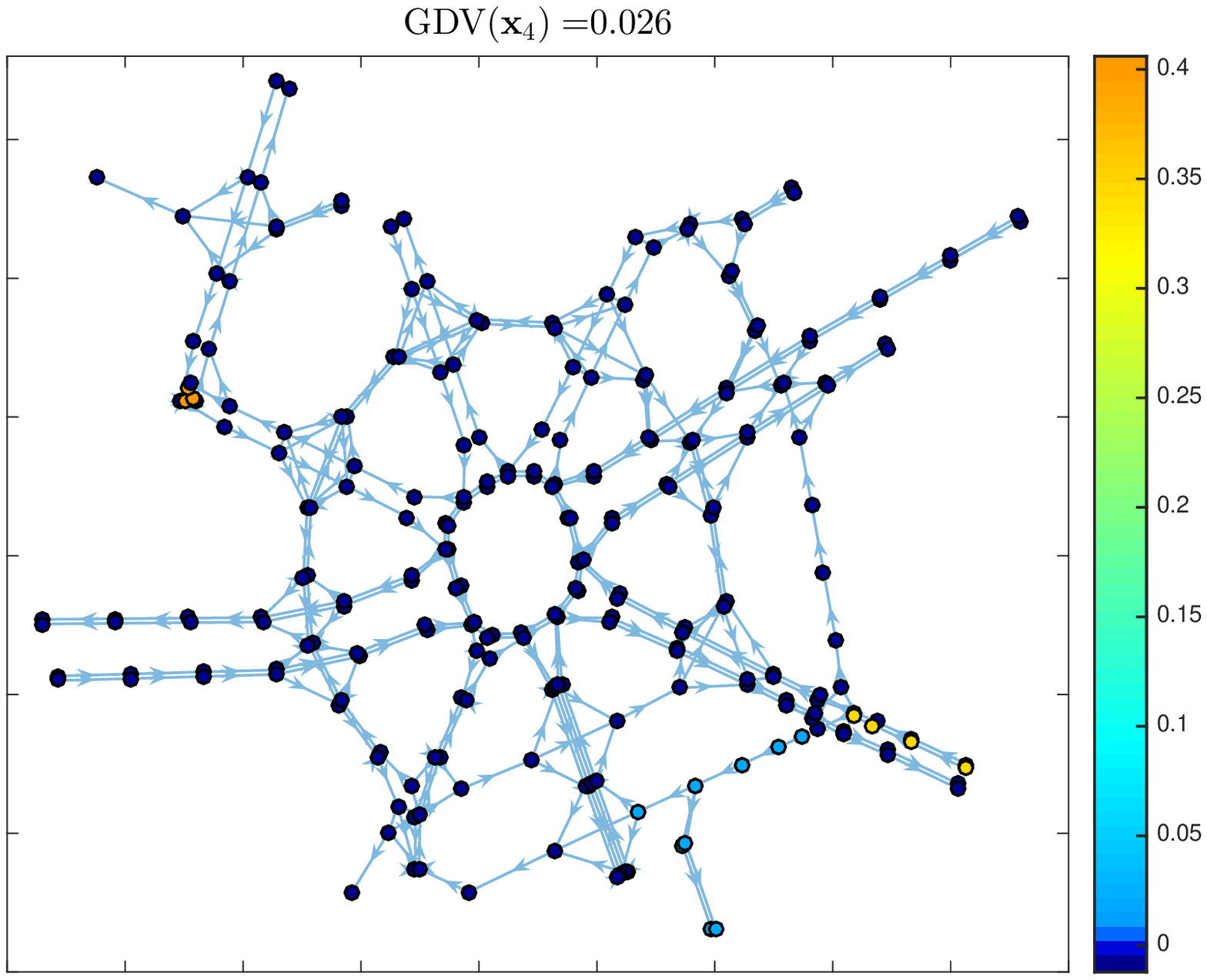}
      \end{subfigure}
\begin{subfigure}[b]{0.5\textwidth}
\centering
              \includegraphics[width=9.4cm,height=7.0cm,bb=35 226 542 600]{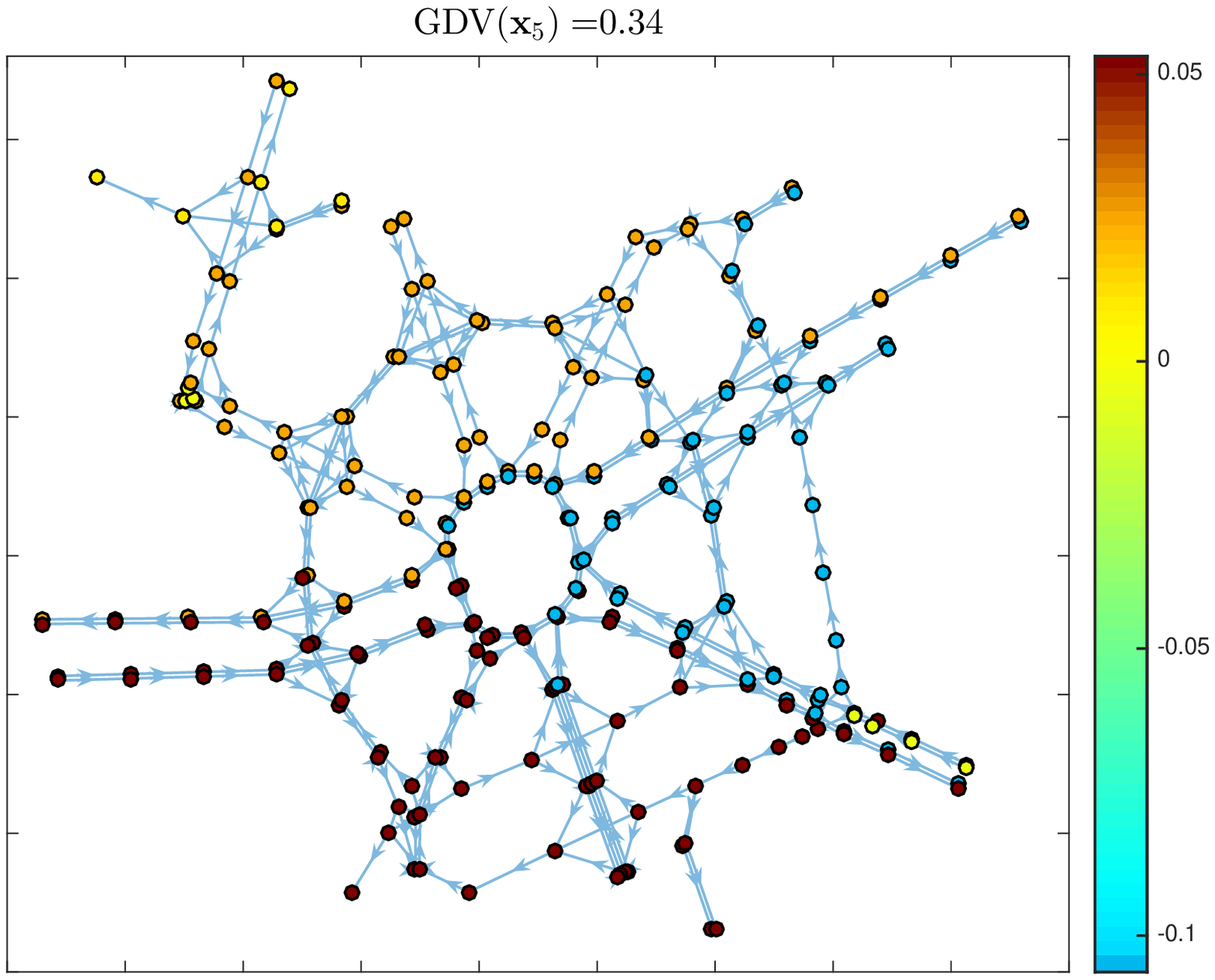}
\end{subfigure}
 \vskip \baselineskip
\begin{subfigure}[b]{0.49\textwidth}\hspace{-1.5cm}
\centering
               \includegraphics[width=9.4cm,height=7.0cm,bb=50 226 542 600]{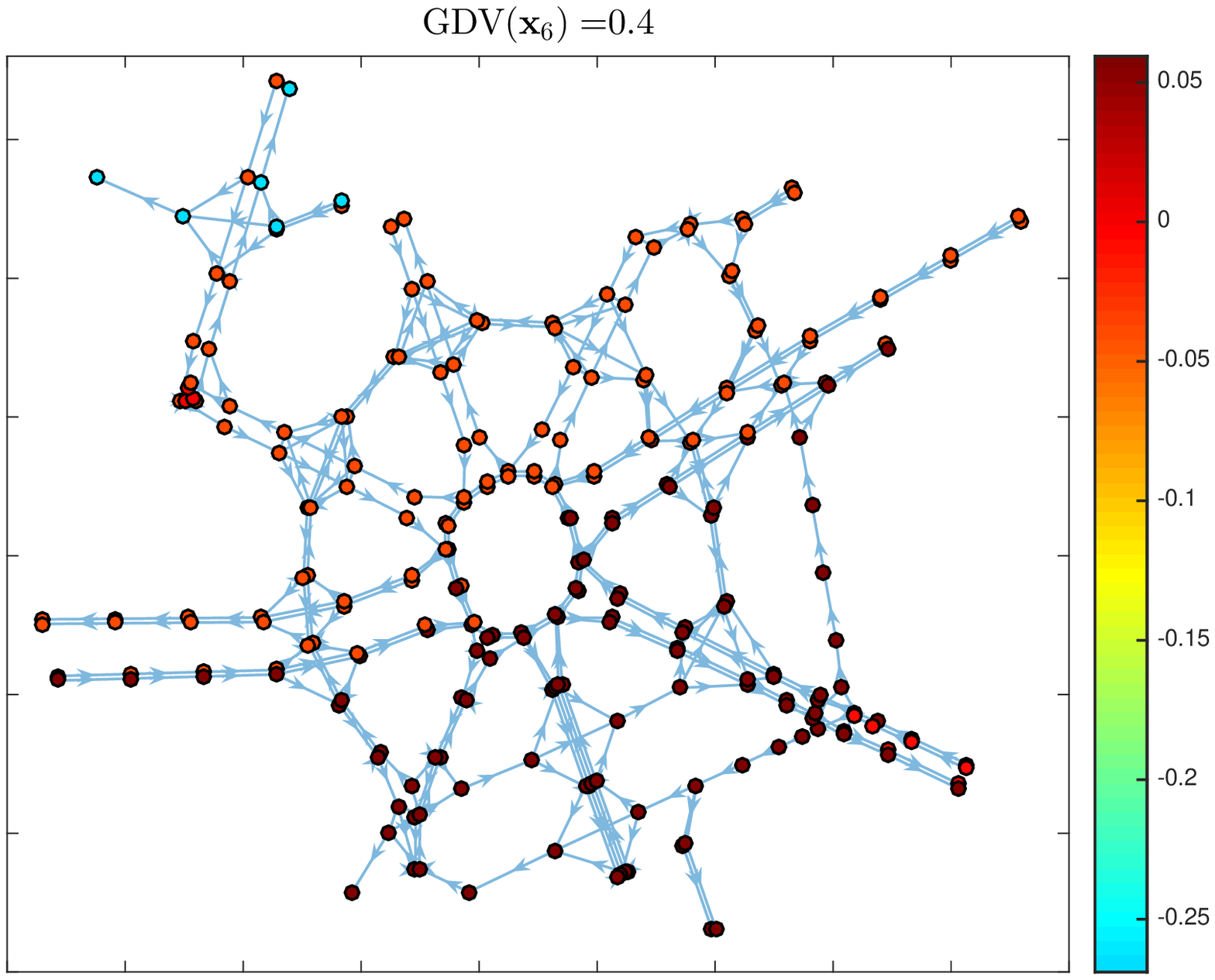}
      \end{subfigure}
\begin{subfigure}[b]{0.5\textwidth}
\centering
              \includegraphics[width=9.43cm,height=7.0cm,bb=34 226 542 600]{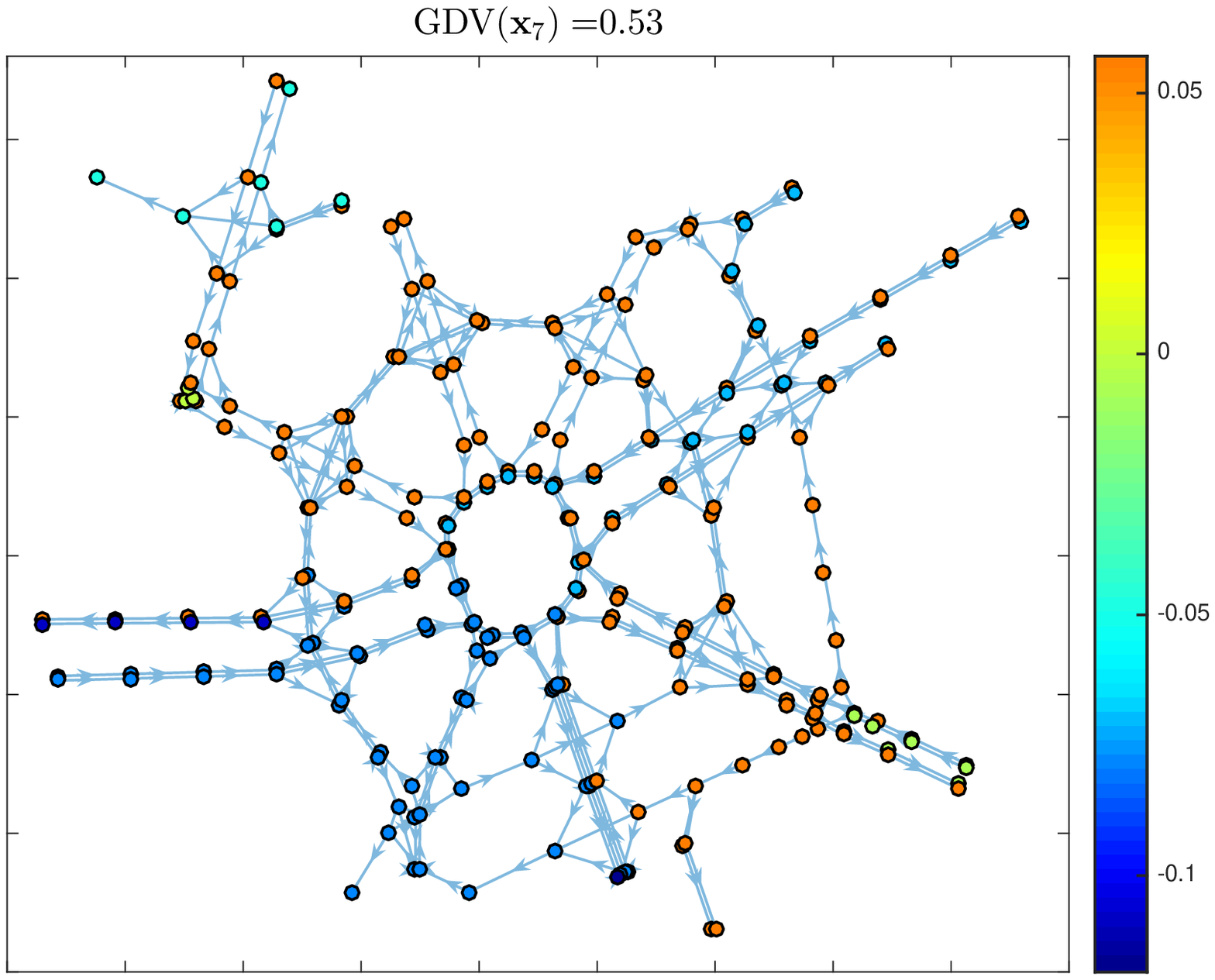}
\end{subfigure}
\vspace{0.2cm}
\caption{Optimal basis vectors $\mx_k$, $k=2,\ldots, 7$  for Algorithm $4$ and the graph in Fig. \ref{fig:figmazzini}.}\label{fig:fdirsquarebal1}
\end{figure*}

\appendix \vspace{-0.1cm}
\subsection{Closed-form solution for problem $\tilde{\mathcal{Q}}_{k,n}$}
\label{A:closed_form} \vspace{-0.1cm}
In this section we provide a closed-form solution for the non-convex problem $\tilde{\mathcal{Q}}_{k,n}$. This problem can
be equivalently written as \vspace{-0.3cm}
\beq
\hspace{-0.5cm}
\begin{split}
\label{sub_probs3}
\mP^{k,n} = \hspace{0.5cm} & \underset{\mP \in \mathbb{R}^{N \times N}}{{\rm arg}\min} \, \hspace{0.5cm} g_{k,n-1}(\mP) \\
 & \hspace{0.3cm} \text{s.t.} \hspace{1.1cm} \mP^T \mP= \mI
\end{split}
\eeq
where $g_{k,n-1}(\mP)\triangleq \langle \mLambda^k, \mP-\mX^{k,n-1} \rangle+\frac{\rho^k}{2}\| \mP-\mX^{k,n-1}\|^{2}_{F} + \frac{c_2^{k,n-1}}{2} \parallel \mP-\mP^{k,n-1}\parallel^{2}_{F}$.
Our proof consists of two steps: i) first, we find
the stationary solutions by solving the KKT necessary conditions; ii) then, we prove that the resulting closed-form solution is a global minimum of the non-convex problem   (\ref{sub_probs3}).
The Lagrangian function $\mathcal{L}_{P}$ associated to (\ref{sub_probs3}) can be written as \vspace{-0.3cm}
\beq
\begin{split}
\mathcal{L}_{P}=& \langle \mLambda^k, \mP-\mX^{k,n-1} \rangle+\frac{\rho^k}{2}\| \mP-\mX^{k,n-1}\|^{2}_{F}\\ & + \frac{c_2^{k,n-1}}{2} \parallel \mP-\mP^{k,n-1}\parallel^{2}_{F}+ \langle \mLambda_1, \mP^T \mP-\mI \rangle
\end{split}\vspace{-0.2cm}
\eeq
where $\mLambda_1 \in \mathbb{R}^{N \times N}$ is the multipliers' matrix associated to the orthogonality constraint.
The KKT conditions become then\vspace{-0.01cm}
\beq\vspace{-0.01cm}
\label{KKT}
\begin{array}{ll}
\begin{split}
\text{a)} \; \nabla_{P}\mathcal{L}_{P}= &\mP[ \mI (\rho^k+c_2^{k,n-1})+2 \mLambda_1]-c_2^{k,n-1} \mP^{k,n-1}\\ &\! -\rho^k \mX^{k,n-1}+\mLambda^k=\mathbf{0}, \end{split}\medskip\\
\;\text{b)}\;  \mLambda_1  \perp  \mP^T \mP-\mI=\mathbf{0}\vspace{-0.2cm}
\end{array}
\eeq
where we  chose  $\mLambda_1=\mLambda_1^T$.
Hence,  defining $\mB \triangleq \mI+2 \mLambda_1/(\rho^k+c_2^{k,n-1})$, from equation a) one gets:\vspace{-0.01cm}
\beq
\label{PBW}
\mP \mB=\mF
\eeq
with $\mF\triangleq \ds \frac{c_2^{k,n-1} \mP^{k,n-1}+\rho^k \mX^{k,n-1}-\mLambda^k}{\rho^k+c_2^{k,n-1}}$.
Let $\mQ \mSigma \mT^T$ be the SVD decomposition of $\mF$.
From (\ref{PBW}), it turns out \vspace{-0.01cm}
\beq \label{PBW1}
\mP \mB=\mQ \mSigma \mT^T
\eeq
and,  using the orthogonality condition b) in (\ref{KKT}),
it holds
\beq
\mB^T \mB= \mT \mSigma^2 \mT^T \; \Rightarrow \; \mB= \mT \mSigma \mT^T.
\eeq
Therefore, replacing $\mB$ in (\ref{PBW1}), we get
\beq
\mP  \mT \mSigma \mT^T =\mQ \mSigma \mT^T \; \Rightarrow\; \mP= \mQ  \mT^T.
\eeq
It remains to prove that $\mP^{\star}=\mP^{k,n}= \mQ  \mT^T$ is  a global minimum for problem (\ref{sub_probs3}).
To this end, it is sufficient to show that
\beq
g_{k,n-1}(\mP^{\star})\leq g_{k,n-1}(\mP), \quad \; \forall\, \mP \, : \, \mP^T\mP=\mI
\eeq
i.e., using the equalities $\parallel \mP^{\star}\parallel_{F}^{2}=\parallel \mP\parallel_{F}^{2}=N$, we have to prove that
$\forall\, \mP \, : \, \mP^T\mP=\mI$, it results
\beq \label{ineq_to_pr}
\begin{split}
&\tr(\mP^{\star  T}(\mLambda^k-\rho^k \mX^{k,n-1}-c_2^{k,n-1}\mP^{k,n-1} )) \leq\\
&\tr(\mP^{T}(\mLambda^k-\rho^k \mX^{k,n-1}-c_2^{k,n-1}\mP^{k,n-1} )).
\end{split}
\eeq
Using the above definition of $\mF$,  (\ref{ineq_to_pr}) reduces to
\beq
\tr(\mP^{\star  T}\mF)  \geq
\tr(\mP^{T} \mF), \quad \; \forall\, \mP \, : \, \mP^T\mP=\mI
\eeq
and since $\mP^{\star}=\mQ  \mT^T$, the final inequality to hold true is
\beq \label{last_in}
\tr(\mSigma)  \geq
\tr(\mT^{T} \mP^T \mQ \mSigma), \quad \; \forall\, \mP \, : \, \mP^T\mP=\mI.
\eeq
Define $\mZ^T:=\mT^{T} \mP^T \mQ$ so that $\mZ^T \mZ=\mI$. Then, from (\ref{last_in}) we get
\beq \label{last_in1}
\tr(\mSigma)  \geq
\tr(\mZ^T \mSigma), \quad \; \forall\, \mZ \, : \, \mZ^T\mZ=\mI.
\eeq
This last inequality holds because $\Sigma_{ii}>0$ and $Z_{ii}  \leq  \mid Z_{ii}\mid  \;\;\leq 1$, $\forall i$, where the latter is implied by $\mZ^T\mZ=\mI$ \cite{Manton}.
Additionally, $Z_{ii}=1$, $\forall i$, if and only if $\mZ=\mI$,
so that the equality in (\ref{last_in1}) holds if and only if $\mZ=\mI$ or $\mP^{\star}=\mQ  \mT^T$.
\vspace{-0.3cm}

\subsection{Proof of Theorem \ref{thm:Th1}}
\label{B:proof Th1} \vspace{-0.1cm}
For lack of space, we omit here the details of the proof, which proceeds using similar arguments  as   in  the proof of  Proposition $2.5$ in \cite{Chen}.
However, to invoke this correspondence, we need to prove that the following properties hold true: i) the function $\mathcal{L}_k$ in (\ref{Lagrange}) satisfies the Kurdyka-{\L}ojasiewicz (K-{\L}) property;
ii) $\mathcal{L}_k$ is  a coercive function.
 To prove point i), let us first introduce some definitions \cite{Bochnak}.
 \begin{definition} \label{semialg_set}
 A semi-algebraic subset of $\mathbb{R}^n$ is a finite union of sets of the  form
 \beq
 \begin{split}
 \{ \bx \in \mathbb{R}^n  : P_1(\bx)=0,\ldots,& P_k(\bx)=0, \\  & Q_1(\bx)>0,\ldots, Q_l(\bx)>0\}
 \end{split}
 \eeq
 where $P_1,\ldots, P_k$ and $Q_1,\ldots, Q_l$ are polynomial in $n$ variables.
 \end{definition}
 \begin{definition} \label{semialg_func}
 A function $f\, : \, \mathbb{R}^n \rightarrow \mathbb{R}$ is said to be semi-algebraic if
 its graph, defined as $\text{gph} f :=\{ (\bx, f(\bx)) | \; \bx \in \mathbb{R}^n\} $, is a semi-algebraic set.
 \end{definition}
It is shown  [\, cf. \cite{Bolte}, Th. $3$] that the semi-algebraic functions satisfy the K-{\L} property.
\begin{definition}
 A function $\phi(\bx)$ satisfies the Kurdyka-{\L}ojasiewicz (K-{\L}) property
 at point $\bar{\bx} \in \text{dom}(\partial \phi)$ if there exists $\theta \in [0,1)$ such that
 \beq
 \ds \frac{|\phi(\bx)-\phi(\bar{\bx})|^{\theta}}{\text{dist}(\mathbf{0}, \partial \phi(\bx))}
 \eeq
 is bounded around $\bar{\bx}$.
 \end{definition}

The global convergence of the PAM method established in \cite{Attouch}
requires the objective function to satisfy the K-{\L} property.
Define $\mW:=(\mX,\mP)$ and
consider the function $\mathcal{L}_k$ in (\ref{Lagrange}), i.e. \vspace{-0.3cm}
\beq \label{Lagrange1}
 \mathcal{L}_k(\mW)=\mathcal{L}(\mX,\mP,\mLambda^k; \rho^k)=f_1(\mX) + f_2(\mP)+g_k(\mX,\mP)
 \eeq
 where $f_1(\mX)=\text{GDV}(\mX)$, $f_2(\mP)=\delta_{\mathcal{S}_t}(\mP)$ and $g_k(\mX,\mP)=\langle \mLambda^k, \mP-\mX \rangle+\frac{\rho^k}{2}\| \mP-\mX\|^{2}_{F}$.
 Observe that $f_1(\mX)=\ds \sum_{i,j=1}^{N} a_{ji}\text{max}(x_i-x_j,0)$ is the weighted sum of  the functions
 $f_{ij}(x_i,x_j)=\text{max}(x_i-x_j,0)$.
 Being a finite sum of semi-algebraic functions also a semi-algebraic function, it is sufficient to show that $f_{ij}$
 is semi-algebraic.
 Assume, w.l.o.g. $y_{ij}=x_i-x_j$ so that $z=f_{ij}(y_{ij})=\text{max}(y_{ij},0)$. The graph of $f_{ij}$ becomes \vspace{-0.1cm}
 \beq \nonumber
\text{gph} f_{ij}\!\!=\!\{ (y_{ij},z)  :  z=y_{ij}, y_{ij}\geq 0 \} \cup \{ (y_{ij},z)  : z=0, y_{ij}\leq 0 \}
 \eeq
 and according to Definition \ref{semialg_set} it is a semi-algebraic set. Then $f_1(\mX)$ as sum of semi-algebraic functions is also semi-algebraic.
 Since $f_2(\mP)$ and $g_k(\mX,\mP)$ are  semi-algebraic functions it follows that $\mathcal{L}_k(\mW)$ is also semi-algebraic.
 It remains to prove point ii) to assess that $\mathcal{L}_k$ is a coercive function, i.e. $\mathcal{L}_k(\mW)\rightarrow \infty$
 when $\|\mW\|_{\infty} \rightarrow \infty$.
 Clearly, the term $f_2(\mP)$ is coercive. The remaining terms in  (\ref{Lagrange1}) can be written as \vspace{-0.1cm}
 \beq \nonumber
 \begin{split}
 f_1(\mX) + g_k(\mX,\mP)\!=\text{GDV}(\mX)+ \ds \frac{\rho^k}{2} \langle \mX, \mX  \rangle \!-\!\langle \rho^k \mP+\mLambda^k, \mX  \rangle\\
 \!\!+ \langle \mLambda^k, \mP\rangle+\ds \frac{\rho^k}{2} \parallel \mP\parallel_{F}^{2}.
 \end{split}
 \eeq
 Since $\mP \in \mathcal{S}_t$ it holds $\parallel \mP\parallel_{F}^{2}=N$.
 Thus, from the inequalities $\langle \mA,\mB \rangle \geq - \parallel \mA \parallel_{F} \parallel \mB \parallel_{F}$ and
 $\parallel \mB \parallel_{F} \leq \parallel \mB \parallel_{1}$, it holds $\langle \mLambda^k, \mP\rangle\geq  - \sqrt{N} \parallel \mLambda^k \parallel_{1}$, so that one gets \vspace{-0.1cm}
 \beq \nonumber
 \begin{split}
 f_1(\mX) + g_k(\mX,\mP)\geq \text{GDV}(\mX)+ \ds \frac{\rho^k}{2} \langle \mX, \mX  \rangle-
  \rho^k \parallel  \mX \parallel_{1}\\ -\langle \mLambda^k, \mX  \rangle
 - \sqrt{N} \parallel \mLambda^k \parallel_{1} +\ds \frac{\rho^k N}{2}
\end{split}
 \eeq
 where we used the inequality $\langle \rho^k \mP, \mX  \rangle \leq \rho^k \parallel  \mX \parallel_{1}$.
 Observe that the sequence $\{\rho^k\}_{k \in \mathbb{N}}$ is non-decreasing when $\gamma>1$ so that $\rho^k>\rho^1$.
 Then the function $f_1(\mX) + g_k(\mX,\mP)$ is coercive being   $\text{GDV}(\mX)+ \ds \frac{\rho^k}{2} \langle \mX, \mX  \rangle$
 a positive  function.

\ifCLASSOPTIONcaptionsoff
  \newpage
\fi
\bibliographystyle{IEEEtran}
\bibliography{main}

\end{document}